\documentclass[10pt]{amsart}
\usepackage[margin=1.5in]{geometry}
\usepackage{caption} 
\usepackage{amsmath, amsthm, amssymb} 

\usepackage[dvipsnames]{xcolor}
\usepackage[colorlinks, allcolors=Blue]{hyperref} 
\hypersetup{
    colorlinks = true,
    linkcolor = purple!90!black,
    citecolor = Blue,
}
\usepackage{thmtools}
\usepackage{tikz-cd} 
\usepackage[nameinlink,capitalize]{cleveref} 
\usepackage{mathtools}
\usepackage[shortlabels]{enumitem} 
\usepackage{microtype} 
\usepackage{mathrsfs}
\usepackage{lscape}
\usepackage{float}
\usepackage{multicol}
\usepackage{euscript}
\newcommand{\euscr}[1]{\EuScript{#1}}
\newtheorem{conj}{Conjecture}[section]
\usepackage[
    backend=biber,
    style=alphabetic,
    maxbibnames=99,
    isbn=false,
    url=false,
    doi=false,
    maxalphanames=5,
    minalphanames=4,
]{biblatex}
\usepackage[parfill]{parskip}

\DeclareMathOperator{\modmod}{/\!/}

\newtheorem{letterthm}{Theorem}

\theoremstyle{definition}
\numberwithin{equation}{section}
\newtheorem{thm}[equation]{Theorem}
\newtheorem{lemma}[equation]{Lemma}
\newtheorem{corollary}[equation]{Corollary}
\newtheorem{proposition}[equation]{Proposition}

\newtheorem{remark}[equation]{Remark}
\newtheorem{definition}[equation]{Definition}
\newtheorem{example}[equation]{Example}
\newtheorem{question}{Question}

\begin{document}

\title{Rings of cooperations for hermitian K-theory over finite fields}
\author{Jackson Morris}
\address{Department of Mathematics, University of Washington, Seattle, Washington}
\email{\href{mailto:jackmann@uw.edu}{jacksonmorris1999@gmail.com}}

%

\begin{abstract}
	We compute the ring of cooperations  $\pi_{*,*}^{\mathbb{F}_q}(\text{kq} \otimes \text{kq})$ for the very effective hermitian K-theory over all finite fields $\mathbb{F}_q$ where $\text{char}(\mathbb{F}_q) \neq 2.$ To do this, we use the motivic Adams spectral sequence and show that all differentials are determined by the integral motivic cohomology of $\mathbb{F}_q$. As an application, we compute the $\mathrm{E}_1$-page of the kq-resolution.
\end{abstract}

\maketitle

\begin{center}
    \includegraphics[scale=.25]{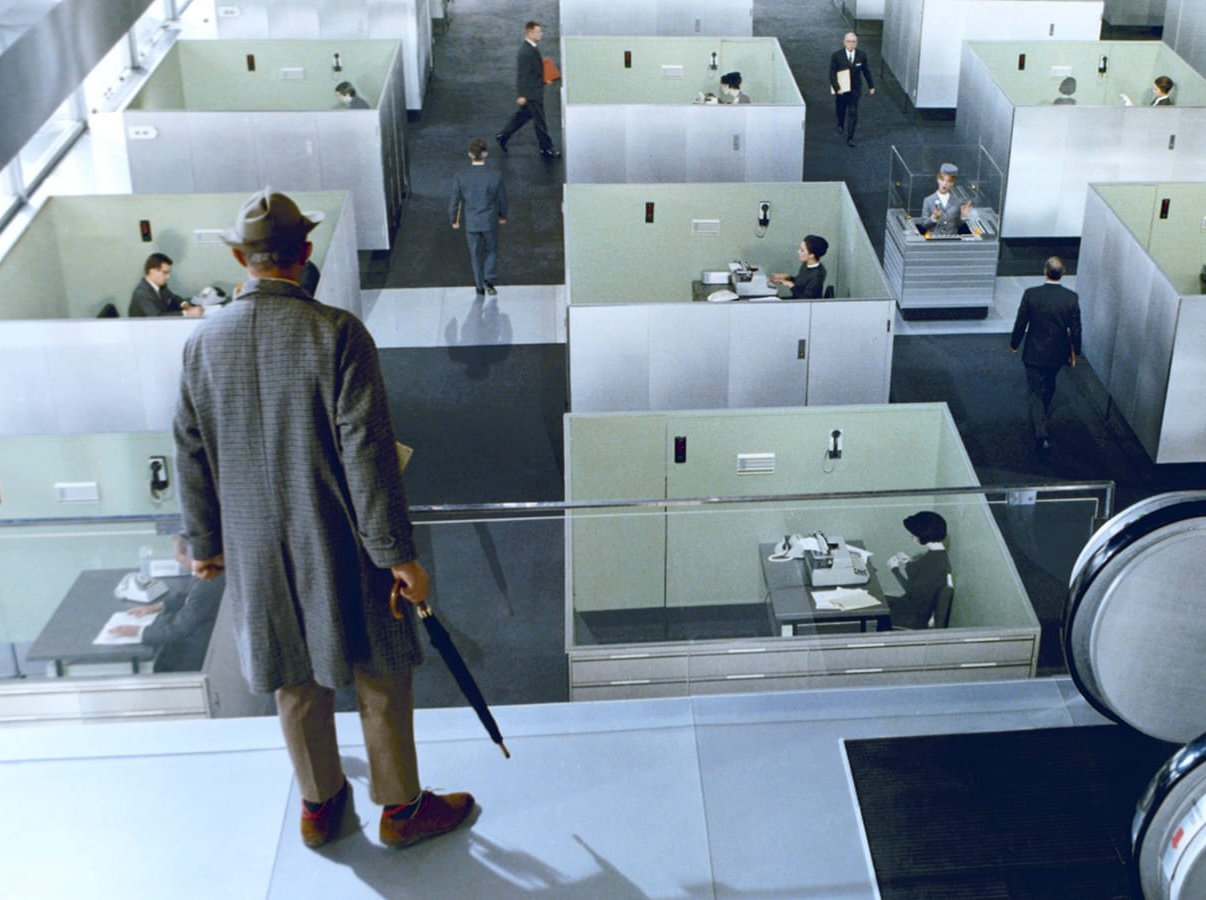} \\
    
    Offices from Jacques Tati's \emph{Playtime} (1967).
\end{center}
\vspace*{\fill}

\newpage
\tableofcontents

\section{Introduction}
The stable motivic homotopy groups of spheres are among the most important invariants in stable motivic homotopy theory. For $F$ a field, the bigraded ring $\pi_{*,*}^F(\mathbb{S})$ has many connections with both topological and arithmetic information regarding $F$. For example, it is a theorem of Morel \cite{MorelKMW} that there is an isomorphism
\[\pi_{-n,-n}^F(\mathbb{S}) \cong \text{K}_{n}^{\text{MW}}(F),\]
where $\text{K}^{\text{MW}}_n(F)$ is the $n^{th}$ Milnor-Witt K-theory of $F$.  As another example, take $F$ to be algebraically closed, and let $e = \text{char}(F)$ for $F$ of positive characteristic and $e=1$ for $F$ of characteristic 0. Then there is an isomorphism after inverting the characteristic \cite{LevineComparison, WO-finite}
\[(\pi_{n,0}^F(\mathbb{S}))[e^{-1}] \cong (\pi_n(\mathbb{S}^{\mathrm{top}}))[e^{-1}],\]
where the right hand side denotes the classical stable homotopy groups of spheres.

An extremely useful tool for computing stable motivic homotopy groups is the motivic Adams spectral sequence. For any motivic ring spectrum E, there is an E-based motivic Adams spectral sequence computing $\pi_{*,*}^F(\mathbb{S})$ up to some completion. Perhaps the most well studied is the $\text{H}\mathbb{F}_p$-based motivic Adams spectral sequence, where $\text{H}\mathbb{F}_p$ is the motivic ring spectrum representing motivic cohomology with mod-$p$ coefficients. We will refer to this spectral sequence simply as the $\textbf{mASS}_p^F(\mathbb{S})$, which has signature
\[\mathrm{E}^{s,f,w}_2=\text{Ext}^{s,f,w}_{\euscr{A}^\vee_p}(\mathbb{M}_p^F, \mathbb{M}_p^F) \implies \pi_{s, w}^F(\mathbb{S}_{p,\eta}^\wedge).\]
This spectral sequence has been well studied over algebraically closed fields \cite{DImASS, HKO-convergencemASS}, the real numbers \cite{BelIsa-stem}, and, following the computation of the dual Steenrod algebra $\euscr{A}^\vee_p$ in positive characteristic \cite{HKO17}, over finite fields \cite{WO-finite}.

Another perspective which one can take to gain insight into $\pi_{*,*}^F(\mathbb{S})$ is by organizing its elements into periodic families. This is the heart of the chromatic approach to motivic homotopy theory. Instead of using the $\textbf{mASS}^F_p(\mathbb{S})$ to compute the motivic stable stems one stem degree at a time, one can use different E-based motivic Adams spectral sequences to localize attention to particular infinite periodic families in $\pi^F_{*,*}(\mathbb{S})$. One downside to this philosophy is that often these spectral sequences are less computable at the onset, and some amount of genuine work must be put into computing the $\mathrm{E}_1$-page. These spectral sequences take the form
\[\mathrm{E}_1 = \pi_{s+f, w}^F(\text{E} \otimes \overline{\text{E}}^{\otimes f}) \implies \pi^F_{s,w}(\mathbb{S}_\mathrm{E}^\wedge),\]
where $\overline{\text{E}}$ is the cofiber of the unit map $\mathbb{S} \to \text{E}$ and $\mathbb{S}_\mathrm{E}^\wedge$ denotes the $\mathrm{E}$-nilpotent completion of $\mathbb{S}$. In practice, it is often easier to compute $\pi_{*,*}^F(\text{E} \otimes \text{E})$, known as the ring of cooperations, and then bootstrap up to the $\mathrm{E}_1$-page.

While this is common in classical stable homotopy theory (see \cite{Mah81, GonzalesBP1ASS, BBBCX}), in motivic homotopy theory this idea is rather new. For instance, the study of $v_1$-periodicity via these techniques is in its nascent stage (although there is work on the subject via the slice spectral sequence due to \cite{belmontisaksenkong-v1R, kongquigley}). Two motivic spectra which are good candidates for accessing $v_1$-periodicity, in that they contain some power of $v_1$ in their homotopy, are the very effective hermitian K-theory spectrum kq \cite{ARO20} and the truncated motivic Brown--Peterson spectrum $\text{BPGL} \langle 1 \rangle$ \cite{Hu-Kriz-remarks}. In \cite{CQ21}, Culver and Quigley study the kq-based Adams spectral sequence, known as the kq-resolution, over $F=\mathbb{C}$ at the prime 2. As an application, they determine the $v_1$-periodic elements in $\pi_{*,*}^\mathbb{C}(\mathbb{S})$. In \cite{Realkqcoop}, we extend the study of the kq-resolution to $F = \mathbb{R}$ by computing the ring of cooperations $\pi_{*,*}^{\mathbb{R}}(\text{kq} \otimes \text{kq})$ up to $v_1$-torsion. In \cite{MorPetTat-BPGL1}, joint work with of the author, Petersen, and Tatum studies the $\text{BPGL} \langle 1 \rangle$-based motivic Adams spectral sequence. We compute the ring of cooperations and produce spectrum-level splittings over the fields $F=\mathbb{C}, \mathbb{R}, \mathbb{F}_q$ and at all primes $p$ where $\text{char}(\mathbb{F}_q) \neq p$.

\subsection{Main Results}
The goal of this paper is to extend the study of the kq-resolution to base fields of positive characteristic. Let $\mathbb{F}_q$ be a finite field where $\text{char}(\mathbb{F}_q) \neq 2$. Our first result is a computation of the ring of cooperations $\pi_{*,*}^{\mathbb{F}_q}(\text{kq} \otimes \text{kq})$ up to $v_1$-torsion.

\begin{letterthm}[\Cref{thm:easyExtB0k}, \Cref{thm:Ext-A1-B0k-F3}]
    The $\textup{\textbf{mASS}}^{\mathbb{F}_q}(\textup{kq} \otimes \textup{kq})$ has signature
    \[\textup{E}_2^{s,f,w} = \bigoplus_{k \geq 0}\Sigma^{4k, 2k}\textup{Ext}_{\euscr{A}(1)^\vee}^{s,f,w}(\mathbb{M}_2^{\mathbb{F}_q}, B_0^{\mathbb{F}_q}(k)) \implies \pi_{s,w}^{\mathbb{F}_q}(\textup{kq} \otimes \textup{kq}),\]
    where $B_0^{\mathbb{F}_q}(k)$ denotes the $k^{th}$ integral motivic Brown--Gitler comodule. We describe the $\textup{E}_\infty$-page, modulo $v_1$-torsion, as a module over $\pi_{*,*}^{\mathbb{F}_q}\textup{(kq)}.$
\end{letterthm}

Contrary to the classical, $\mathbb{C}$-motivic, and the $\mathbb{R}$-motivic analogues of this question, this spectral sequence does not collapse on the $\mathrm{E}_2$-page. Since $B_0^{\mathbb{F}_q}(0) \cong \mathbb{M}_2^{\mathbb{F}_q}$, our description of the $\mathrm{E}_2$-page identifies the $k=0$-summand as the $\mathrm{E}_2$-page of the $\textbf{mASS}^{\mathbb{F}_q}(\text{kq})$, which also does not collapse at the $\mathrm{E}_2$-page. We make the following observation, inspired by findings of Ormsby \cite{Ormsby11}, Ormsby and \O stv\ae r \cite{Ormsby-Ostvaer-motivicbp}, as well as of the author, Petersen, and Tatum \cite{MorPetTat-BPGL1}.

\begin{letterthm}[\Cref{thm:kqsmashkqdifsF5}, \Cref{thm:kqsmashkqdifsF3}]
    The differentials in the $\textup{\textbf{mASS}}^{\mathbb{F}_q}(\textup{kq} \otimes \textup{kq})$ are determined by the $\textup{Ext}_{\euscr{A}(1)^\vee}^{*,*,*}(\mathbb{M}_2^{\mathbb{F}_q}, \mathbb{M}_2^{\mathbb{F}_q})$-module structure of the $\textup{E}_2$-page and the differentials of the $\textup{\textbf{mASS}}^{\mathbb{F}_q}(\textup{kq})$.
\end{letterthm}

The differentials in the $\textbf{mASS}^{\mathbb{F}_q}(\text{kq})$ are themselves lifted from the $\textbf{mASS}^{\mathbb{F}_q}(\text{H}\mathbb{Z})$ along the natural quotient map $\text{kq} \to \text{H}\mathbb{Z}$. The above theorem implies that relative to the algebra $\text{Ext}^{*,*,*}_{\euscr{A}(1)^\vee}(\mathbb{M}_2^{\mathbb{F}_q}, \mathbb{M}_2^{\mathbb{F}_q})$, the only interesting part of the ring of cooperations for hermitian K-theory is determined by the integral motivic cohomology of $\mathbb{F}_q$. We are also able to use this observation to compute $\pi_{*,*}^{\mathbb{F}_q}(\text{ksp})$ (see \Cref{lem:kspmASSDifsEasy,lemma:kspmASSF3}) reproving a result of Friedlander \cite{Friedlander76}.

As an application, we describe the $n$-line of the $\textup{E}_1$-page of the kq-resolution.

\begin{letterthm}[\Cref{prop:n-lineE2}, \Cref{thm:n-lineDifs}]
\label{thm:C}
    The $\textup{\textbf{mASS}}^{\mathbb{F}_q}(\textup{kq} \otimes \overline{\textup{kq}}^{\otimes n})$ has signature
    \[\textup{E}_2^{s,f,w} = \bigoplus_{K \in \euscr{K}_n} \Sigma^{4|K|, 2|K|}\textup{Ext}_{\euscr{A}(1)^\vee}^{s,f,w}(\mathbb{M}_2^{\mathbb{F}_q}, B_0^{\mathbb{F}_q}(K)) \implies \pi_{s, w}^{\mathbb{F}_q}(\textup{kq} \otimes \overline{\textup{kq}}^{\otimes n}),\]
    where $\euscr{K}_n = \{K = (k_1, \dots, k_n): k_j \geq 1 \textup{ for all } j\}$, $|K| = \sum_{j=1}^nk_j$, and $B_0^{\mathbb{F}_q}(K) = \bigotimes_{j=1}^nB_0^{\mathbb{F}_q}(k_j).$ For $n=0$, the differentials are given by the differentials in the $\textup{\textbf{mASS}}^{\mathbb{F}_q}(\textup{kq})$. For $n >0$, the differentials are determined by the underlying module structure over the $0$-line.
\end{letterthm}

We view the kq-resolution as the key tool to understanding $v_1$-periodicity in $\pi_{*,*}^{\mathbb{F}_q}(\mathbb{S})$, and we will use \Cref{thm:C} in future work to analyze the kq-resolution over $\mathbb{R}$ and $\mathbb{F}_q$.

\subsection{Organization}
In \Cref{section:background}, we introduce relevant background. We recall the stable motivic homotopy category $\text{SH}(\mathbb{F}_q)$, the motivic Adams spectral sequence, the dual Steenrod algebra and Brown--Gitler comodules, and the algebraic Atiyah--Hirzebruch spectral sequence. In \Cref{section:trivial}, we compute the ring of cooperations $\pi_{*,*}^{\mathbb{F}_q}(\text{kq} \otimes \text{kq})$ in the case where $q \equiv 1 \,(4)$. In \Cref{section:nontrivial}, we compute the ring of cooperations $\pi_{*,*}^{\mathbb{F}_q}(\text{kq} \otimes \text{kq})$ in the case where $q \equiv 3 \, (4)$. Many of the arguments in \Cref{section:trivial} carry over directly to \Cref{section:nontrivial}, but the particular details of the algebra involved are quite different. As such, we separate these cases into two sections. In \Cref{section:kqres}, we apply our findings to deduce results about the kq-resolution.

\subsection{Notation and Conventions}
\begin{itemize}
    \item We let $\mathbb{F}_q$ denote the field with $q$ elements where $\text{char}(\mathbb{F}_q) \neq 2$.
    \item We work in the stable, presentably symmetric monoidal $\infty$-category $\text{SH}(\mathbb{F}_q)$. We let 
    \[
    \mathbb{S} := \Sigma^{\infty}_{\mathbb{P}^1}\text{Spec}(\mathbb{F}_q)_+
    \]
    denote the monoidal unit and use $\otimes :=\otimes_{\mathbb{S}}$ for the monoidal product.
    \item We let $\text{H}_{*,*}(-)$ denote mod-2 motivic homology and let $\mathbb{M}_2^{\mathbb{F}_q}$ denote the mod-2 motivic homology of a point.    
    \item For $B$ any Hopf algebra over $\mathbb{M}_2^{\mathbb{F}_q}$ and $M$ any $B$-comodule, we use the common abbreviation
    \[\text{Ext}^{s,f,w}_B(M) := \text{Ext}^{s,f,w}_B(\mathbb{M}_2^{\mathbb{F}_q}, M)\]
    to denote the cohomology of $B$ with coefficients in $M$. Note that this Ext is taken in the category $\text{Comod}(B)$. Typically, $B$ will be a Hopf algebra related to the dual Steenrod algebra $\euscr{A}^\vee$.
    \item Our grading convention for Ext is $(s,f,w)$, where $s$ is the stem (or total) degree, $f$ is the Adams filtration, and $w$ is the motivic weight. This implies that Adams differentials take the form
    \[d_r:\mathrm{E}_r^{s,f,w} \to \mathrm{E}_r^{s-1, f+r, w}.\]
    We also refer the the coweight, which is $cw=s-w$. Note that Adams differentials decrease coweight by 1. We display all charts in $(s,f)$-grading with weight suppressed.
    \item Throughout, all spectra are implicitly 2-complete.
\end{itemize}

\subsection{Acknowledgments}
The work presented here constitutes a part of the author's thesis. The author thanks their advisors, Kyle Ormsby and John Palmieri, for their patience and guidance throughout their PhD. The author thanks J.D. Quigley for many conversation regarding kq-resolutions, and thanks Guchuan Li, Sarah Petersen, and Liz Tatum for valuable discussions influencing this project. The author also thanks Gijs Heuts for an inspirational talk at European Talbot 2025 that inspired the author to include \Cref{question} and \Cref{conjecture}. Finally, the author extends his gratitude to Madeline Borowski, Jay Reiter, and Alex Waugh for their willingness to listen to the author ramble on this topic, and thanks Cameron Winter for inspiration \cite{HeavyMetal}.

\section{Background}
\label{section:background}
Throughout, let $\mathbb{F}_q$ be a finite field with $\text{char}(\mathbb{F}_q) \neq 2.$ In this section, we review motivic cohomology and hermitian K-theory, the dual Steenrod algebra, and Brown--Gitler comodules. Then, we describe the $\text{H}\mathbb{F}_2$- and kq-based motivic Adams spectral sequences and outline our program to compute the ring of cooperations.

\subsection{Motivic cohomology and hermitian K-theory}
Spitzweck constructed the \textit{integral motivic cohomology spectrum} $\text{H}\mathbb{Z} \in \text{SH}(\mathbb{Z})$ representing motivic cohomology with integral coefficients \cite{Spitzweck-HZ-overdedekind}. The unique map $f:\mathbb{Z} \to \mathbb{F}_q$ induces a pullback $f^*:\text{SH}(\mathbb{F}_q) \to \text{SH}(\mathbb{Z})$; we abusively denote $\text{H}\mathbb{Z} := f^*(\text{H}\mathbb{Z}) \in \text{SH}(\mathbb{F}_q)$ the integral motivic cohomology spectrum over $\mathbb{F}_q$. This represents motivic cohomology in the sense that if $X \in \text{Sm}_{\mathbb{F}_q}$, then 
\[[X, \Sigma^{s,w}\text{H}\mathbb{Z}] = \text{H}^{s,w}(X; \mathbb{Z}),\]
where we let $X$ also denote the motivic suspension spectrum of $X$. Of particular interest to us is when $X = \text{Spec}(\mathbb{F}_q)$. As we will be computing with the 2-primary motivic Adams spectral sequence, we are only concerned with the 2-completion $\text{H}^{p,q}(\mathbb{F}_q;\mathbb{Z})_2^\wedge$. These groups were calculated by Soul\'e \cite{Soule79}:
\begin{equation}
\label{eq:MotivicCohomology}
\text{H}^{-s,-w}(\mathbb{F}_q;\mathbb{Z})_2^\wedge = \left\{\begin{array}{rl}
    \mathbb{Z}_2 & s=w=0 \\
    \mathbb{Z}/(q^{-w}-1)_2 & s=-1, w \leq 1\\
    0 & \text{else}.
\end{array}\right.
\end{equation}
Notice that $\text{H}^{-s,-w}(\mathbb{F}_q;\mathbb{Z}) = \pi_{s, w}^{\mathbb{F}_q}(\text{H}\mathbb{Z})$, so this also calculates the homotopy groups of the integral motivic cohomology spectrum.

There is also a \textit{mod-2 motivic cohomology spectrum} $\text{H}\mathbb{F}_2$.
By \cite{Voemotiviccohomology}, we have that
\[\text{H}^{-s,-w}(\mathbb{F}_q;\mathbb{Z}/2) = \pi_{s, w}^{\mathbb{F}_q}(\text{H}\mathbb{F}_2) :=\mathbb{M}_2^{\mathbb{F}_q} = (\text{K}_*^{\text{M}}(\mathbb{F}_q)/2)[\tau],\]
where $\text{K}^\text{M}_*(\mathbb{F}_q)$ denotes the \textit{Milnor K-theory} of $\mathbb{F}_q$, and where $|\tau|= (0, -1)$ and $|\text{K}^\text{M}_n(\mathbb{F}_q)| = (-n, -n)$. Recall that Milnor K-theory $\text{K}^\text{M}_*(\mathbb{F}_q)$ is defined as the graded tensor algebra over $\mathbb{Z}$ generated by the symbols $[a]$ in degree one for $a \in \mathbb{F}_q^\times$, subject to the relations $[a] \otimes [1-a] = 0$ and $[a]+[b]=[ab]$. It is a routine exercise to show that $\text{K}^{\text{M}}_n(\mathbb{F}_q)$ vanishes in degrees $n \geq 2$. In particular, since every element of $\mathbb{F}_q^\times$ is either a square or a non square, this implies that
\[(\text{K}^\text{M}_0(\mathbb{F}_q)/2) = (\text{K}^{\text{M}}_1(\mathbb{F}_q)/2) = \mathbb{Z}/2.\]
To differentiate between the cases of $q \equiv 1 \, (4)$ and $q \equiv 3 \, (4)$, we let $u$ be any generator of $\text{K}_1^\text{M}(\mathbb{F}_q)$ for $q \equiv 1 \, (4)$, and let $\rho = [-1]$ be a generator of $\text{K}^{\text{M}}_1(\mathbb{F}_q)$ for $q \equiv 3 \, (4)$. Notice that if $q \equiv 1 \, (4)$, then $-1$ has a square root in $\mathbb{F}_q$, hence $[-1] = 0 \in \text{K}^\text{M}_1(\mathbb{F}_q)$, so that $u$ is not represented by $-1$. Then we have identifications: 
\[
\mathbb{M}_2^{\mathbb{F}_q} =  \left\{ \begin{array}{cl}
\mathbb{F}_2[u, \tau]/(u^2) &  q \equiv 1 \, (4)\\
\mathbb{F}_2[\rho, \tau]/(\rho^2) & q \equiv 3 \, (4),
\end{array} \right. 
\] 
where $|\tau| = (0, -1)$ and $|u|=|\rho| = (-1, -1)$.

There is a \textit{hermitian K-theory} spectrum KQ \cite{Hor05, CalHarNar25} which is the main character of our work. The homotopy groups $\pi_{*,*}^{\mathbb{F}_q}(\text{KQ})$ intertwine the classical stable homotopy of the orthogonal, unitary, and symplectic K-groups of $\mathbb{F}_q$. We will only be interested in the orthogonal K-groups 
$\text{KO}_*(\mathbb{F}_q):= \pi_*(\text{BO}(\mathbb{F}_q)^+)$ and the symplectic K-groups $\text{KSp}_*(\mathbb{F}_q):=\pi_*(\text{BSp}(\mathbb{F}_q)^+)$; for more details, see \cite{Hor05}. These are realized by hermitian K-theory in the following way:
\[\pi_{s,w}^{\mathbb{F}_q}(\text{KQ}) = \left\{\begin{array}{rl}
    \text{KO}_{s-2w}{(\mathbb{F}_q}) &  w \equiv 0 \, (4);\\
    \text{KSp}_{s-2w}(\mathbb{F}_q) & w \equiv 2 \, (4).\\
\end{array} \right.\]
The hermitian K-theory spectrum is an $\mathbb{E}_\infty$ motivic ring spectrum \cite{CalHarNar25}. We let kq denote the very effective cover of KQ \cite{ARO20}, which we call the \emph{very effective hermitian K-theory}.
The ring structure on KQ lifts to kq, and the unit map $\mathbb{S} \to \text{kq}$ induces Morel's isomorphism \cite{RSOfirst}
\[\pi_{-n, -n}^{\mathbb{F}_q}(\mathbb{S}) \cong \text{K}_n^{\text{MW}}(\mathbb{F}_q).\]
In particular, $\pi_{0,0}^{\mathbb{F}_q}(\text{kq}) \cong \text{GW}(\mathbb{F}_q)$, the Grothendieck--Witt ring. We will not need all of the details of the very effective slice filtration (see recent work of Bannwart \cite[Section 1]{Bann-realbetti} for a wonderful recollection). For the purposes of our work, one should treat the very effective cover as an analogue of the  connective cover. Indeed, for $s \geq 0$ there is an isomorphism $\pi_{s,w}^{\mathbb{F}_q}(\text{KQ}) \cong \pi_{s,w}^{\mathbb{F}_q}(\text{kq}).$ However, we will see that the homotopy groups of kq do not vanish for $s <0$ (an interpretation of these homotopy groups in negative stem degrees is given in \Cref{rem:f5NegativeStemkq} and \Cref{rem:f3NegativeStemkq}). 

There are cofiber sequences of spectra:
    \[\Sigma^{1,1}\text{kq} \xrightarrow{\eta}\text{kq} \to \text{kgl},\]
    \[\Sigma^{2,1}\text{kgl} \xrightarrow{\beta}\text{kgl} \to \text{H}\mathbb{Z},\]
where kgl denotes the effective algebraic K-theory spectrum,
$\eta$ is the motivic Hopf map, and $\beta$ is the Bott periodicity class \cite{ARO20}. Combining the last morphism of each cofiber sequence gives a map 
\begin{equation}
\label{eq:kqToHZ}
\text{kq} \to \text{H}\mathbb{Z}
\end{equation}
which will feature heavily in our arguments.

Related to hermitian K-theory is the \textit{very effective symplectic K-theory} spectrum ksp, which is defined as the very effective cover of $\Sigma^{4,2}\text{KQ}$. The homotopy groups of ksp are a shifted version of the homotopy groups of kq, which we investigate in more detail in later sections. We note that there is a cofiber sequence
\begin{equation}
\label{cofib:kqTokspToHZ}
\Sigma^{4,2}\text{kq} \to \text{ksp} \to \text{H}\mathbb{Z}
\end{equation}
coming from the very effective slice filtration on $\Sigma^{4,2}\text{KQ}$. For more details on ksp and the very effective slice tower for KQ, see \cite[Section 3]{Realkqcoop}.

The hermitian K-theory of finite fields was calculated by Friedlander \cite[Theorem 1.7]{Friedlander76}. We display the relevant results in \Cref{table:fried}.
\begin{table}[H]
    \centering
    \setlength{\tabcolsep}{0.5em} 
    {\renewcommand{\arraystretch}{1.2}
    \begin{tabular}{|l||l|l|}
        \hline
       $n$ modulo 8 & $\text{KO}_n(\mathbb{F}_q)$  & $\text{KSp}_n(\mathbb{F}_q)$  \\
       \hline
       \hline
       0 & $\mathbb{Z}/2$ & 0\\
       1 & $\mathbb{Z}/2 \oplus \mathbb{Z}/2$ &0\\
       2 & $\mathbb{Z}/2$&0\\
       3 & $\mathbb{Z}/(q^{(n+1)/2}-1)$ &$\mathbb{Z}/(q^{(n+1)/2}-1)$\\
       4 & 0 &$\mathbb{Z}/2$\\
       5 & 0 &$\mathbb{Z}/2 \oplus \mathbb{Z}/2$\\
       6 & 0 &$\mathbb{Z}/2$\\
       7 & $\mathbb{Z}/(q^{(n+1)/2}-1)$ &$\mathbb{Z}/(q^{(n+1)/2}-1)$\\
       \hline
    \end{tabular}}
    \caption{Friedlander's calculation of $\text{KO}_n(\mathbb{F}_q)$ and $\text{KSp}_n(\mathbb{F}_q)$, where $\text{char}(\mathbb{F}_q) \neq 2$ and $n > 0$.}    
    \label{table:fried}
\end{table}

\subsection{The dual Steenrod algebra and Brown--Gitler comodules}
\label{subsec:dualsteenrod}
The \textit{dual Steenrod algebra} $\euscr{A}^\vee := \pi_{*,*}^{\mathbb{F}_q}(\text{H}\mathbb{F}_2 \otimes \text{H}\mathbb{F}_2)$ was calculated in positive characteristic by Hoyois, Kelly, and \O stv\ae r \cite[Proposition 4.3]{HKO17}:
\[\euscr{A}^\vee = \mathbb{M}_2^{\mathbb{F}_q}[\overline{\tau}_0, \overline{\tau}_1, \dots, \overline{\xi}_1, \overline{\xi}_2, \dots]/T,\]
where $|\overline{\tau}_i| = (2^{i+1}-1, 2^i-1)$ and $|\overline{\xi}_i| = (2^{i+1}-2, 2^i-1).$ The ideal $T$ of relations is dependent on the base field $\mathbb{F}_q$. In the cases we are concerned with, we have that
\[T = \left\{\begin{array}{ll}
    (\overline{\tau}_i^2 = \tau \overline{\xi}_{i+1} ) & q \equiv 1 \, (4)\\
    (\overline{\tau}_i^2 = \tau\overline{\xi}_{i+1} + \rho \overline{\tau}_{i+1} + \rho \overline{\tau}_0\overline{\xi}_{i+1}) &  q \equiv 3 \, (4).\\
\end{array}\right.\]
Note that the pair $(\mathbb{M}_2^{\mathbb{F}_q}, \euscr{A}^\vee)$ is a Hopf algebroid rather than a Hopf algebra \cite[Appendix A]{Rav86}.

For $n \geq 0$, let $\euscr{A}(n)^\vee$ be the quotient algebra
\[\euscr{A}(n)^\vee \cong \euscr{A}^\vee/(\overline{\xi}_1^{2^n}, \overline{\xi}_2^{2^{n-1}}, \dots, \overline{\xi}^2_n, \overline{\xi}_{n+1}, \overline{\xi}_{n+2,} \dots, \overline{\tau}_{n+1}, \overline{\tau}_{n+2}, \dots).\]
There is a sub-Hopf algebra of the motivic Steenrod algebra
\[\euscr{A}(n) = \langle \text{Sq}^1, \text{Sq}^2, \hdots, \text{Sq}^{2^n} \rangle;\]
let $(\euscr{A}\modmod \euscr{A}(n))^\vee$ be the subalgebra of $\euscr{A}^\vee$ given by
\[(\euscr{A} \modmod \euscr{A}(n))^\vee = \mathbb{M}_2^{\mathbb{F}_q}[\overline{\xi}_1^{2^n}, \overline{\xi}_2^{2^{n-1}}, \hdots, \overline{\tau}_{n+1}, \hdots]/(\overline{\tau}_i^2 = \rho \overline{\tau}_{i+1}+\rho \overline{\tau_0}\overline{\xi}_{i+1} + \tau \overline{\xi}_{i+1}).\]
These subalgebras naturally arise via motivic homology. One can show that there are $\euscr{A}^\vee$-comodule algebra isomorphisms
\[\text{H}_{*,*}(\text{H}\mathbb{Z}) \cong (\euscr{A} \modmod\euscr{A}(0))^\vee, \quad \text{H}_{*,*}(\text{kq}) \cong (\euscr{A} \modmod \euscr{A}(1))^\vee\] by using the long exact sequence in homology associated to the cofiber sequences \cite{ARO20}
\[\text{H}\mathbb{Z} \xrightarrow{2}\text{H}\mathbb{Z} \to \text{H}\mathbb{F}_2, \quad \Sigma^{1,1}\text{kq} \xrightarrow{\eta}\text{kq} \to \text{kgl}.\]

We next introduce motivic Brown--Gitler comodules. We define the \textit{Mahowald weight filtration} on $\euscr{A}^\vee$ by setting 
\[\text{wt}(\overline{\xi}_i)=\text{wt}(\overline{\tau}_i)=2^i, \quad \text{wt}(\tau)=\text{wt}(\rho)=\text{wt}(u)=0\] and letting $\text{wt}(xy)=\text{wt}(x)+\text{wt}(y)$. This naturally extends to the subalgebras $(\euscr{A} \modmod \euscr{A}(n))^\vee$. The \textit{motivic Brown--Gitler comodule} $B_n^{\mathbb{F}_q}(k)$ is the $\euscr{A}(n)^\vee$-comodule
\[B_n^{\mathbb{F}_q}(k) = \langle x \in (\euscr{A}\modmod \euscr{A}(n))^\vee : \text{wt}(x) \leq 2^{n+1}k\rangle.\]
A key property of the Brown--Gitler comodules that we will use is the following.
\begin{thm}[{\cite[Theorem 3.20]{CQ21}}]
\label{kq brown gitler homology}
    There is an isomorphism of $\euscr{A}(1)^\vee$-comodules
    \[(\euscr{A}\modmod \euscr{A}(1))^\vee \cong \bigoplus_{k \geq 0}\Sigma^{4k, 2k}B_0^{\mathbb{F}_q}(k).\]
\end{thm}
The $0^{th}$ Brown--Gitler comodule is given by $B_0^{\mathbb{F}_q}(0) \cong\mathbb{M}_2^{\mathbb{F}_q}$, and the first Brown--Gitler comodule is given by $B_0^{\mathbb{F}_q}(1) \cong\mathbb{M}_2^{\mathbb{F}_q}\{1, \overline{\xi}_1, \overline{\tau}_1\}$. It was shown in \cite[Example 3.12]{CQ21} that there is a spectrum $\text{H}\mathbb{Z}_1^{\mathbb{F}_q}$ such that
\[\text{H}_{*,*}(\text{H}\mathbb{Z}_1^{\mathbb{F}_q}) \cong B_0^{\mathbb{F}_q}(1)\]
as $\euscr{A}^\vee$-comodules. Additionally, there is an equivalence of spectra up to 2-completion \cite[Proposition 3.3]{Realkqcoop}:
\begin{equation}
\label{eq:kspSplitting}
    \text{ksp} \simeq \text{kq} \otimes \text{H}\mathbb{Z}_1^{\mathbb{F}_q}.
\end{equation}
In general, the question of finding a motivic spectrum realizing a Brown--Gitler comodule is not simple to answer; see \cite[Remark 2.9]{Realkqcoop}.

There are short exact sequences of $\euscr{A}(1)^\vee$-comodules relating motivic Brown--Gitler comodules. We refer the reader to \cite{CQ21} for a proof.

\begin{proposition}[{\cite[Lemma 3.21]{CQ21}}]
\label{prop:ses bg}
    There are short exact sequences of $\euscr{A}(1)^\vee$-comodules:
    \[0 \to \Sigma^{4k, 2k}B_0^{\mathbb{F}_q}(k) \to B_0^{\mathbb{F}_q}(2k) \to B_1^{\mathbb{F}_q}(k-1) \otimes (\euscr{A}(1) \modmod \euscr{A}(0))^\vee \to 0,\]
    \[0 \to \Sigma^{4k, 2k}B_0^{\mathbb{F}_q}(k) \otimes B_0^{\mathbb{F}_q}(1) \to B_0^{\mathbb{F}_q}(2k+1) \to B_1^{\mathbb{F}_q}(k-1) \otimes (\euscr{A}(1) \modmod \euscr{A}(0))^\vee \to 0.\]
\end{proposition}
\begin{remark}
    In \cite{MorPetTat-BPGL1}, we define Brown--Gitler comodules using the Mahowald weight filtration on the subalgebras $(\euscr{A} \modmod \euscr{E}(n))^\vee$. This leads to a different family of subcomodules, except in the case where $n=0$.
\end{remark}

An interesting feature in the mod-2 motivic cohomology of finite fields is that, while $\mathbb{M}_2^{\mathbb{F}_q}$ is consistent across all fields of characteristic different from 2, the \textit{Bockstein homomorphism}, which is the connecting homomorphism associated to the coefficient sequence
\[0 \to \mathbb{Z}/2 \to \mathbb{Z}/4 \to \mathbb{Z}/2 \to 0,\]
acts differently on $\text{H}^{s,w}(\mathbb{F}_q; \mathbb{Z}/2)$ depending on whether $q \equiv 1 \, (4)$ or $q \equiv 3 \, (4)$. Notice that this Bockstein is represented by $\text{Sq}^1$ in the motivic Steenrod algebra \cite{Voemotiviccohomology}. We can phrase the action of the Bockstein in terms of the action of the motivic Steenrod algebra on $\mathbb{M}_2^{\mathbb{F}_q}$.
\begin{proposition}[{\cite{WO-finite}}]
    The action of $\textup{Sq}^1$ on $\mathbb{M}_2^{\mathbb{F}_q}$ is given by $\textup{Sq}^1(\tau)=\rho$. In particular, the action is trivial if and only if $q \equiv 1 \, (4)$.
\end{proposition}
This implies that $\mathbb{M}_2^{\mathbb{F}_q}$ is central in $\euscr{A}^\vee$ when $q \equiv 1 \, (4)$, meaning that $\euscr{A}^\vee$ is a Hopf algebra over $\mathbb{M}_2^{\mathbb{F}_q}$. We will often refer to the case of $q \equiv 1 \, (4)$ as the \textit{trivial Bockstein action} and the case of $q \equiv 3 \, (4)$ as the \textit{nontrivial Bockstein action}. This subtle difference makes a noticeable impact on computation and is the reason we separate our work into distinct sections.

\subsection{Motivic Adams spectral sequence}
Let $\text{E} \in \text{SH}(\mathbb{F}_q)$ be a motivic ring spectrum. There is a canonical Adams tower associated to the unit map $\mathbb{S} \to \text{E}$ taking the form
\[\begin{tikzcd}
	{\mathbb{S}} & {\Sigma^{-1, 0}\overline{\text{E}}} & {\Sigma^{-2, 0}\overline{\text{E}} \otimes \overline{\text{E}}} & \cdots \\
	{\text{E}} & {\Sigma^{-1, 0}\overline{\text{E}} \otimes \text{E}} & {\Sigma^{-2, 0}\overline{\text{E}} \otimes \overline{\text{E}} \otimes \text{E}}
	\arrow[from=1-1, to=2-1]
	\arrow[from=1-2, to=1-1]
	\arrow[from=1-2, to=2-2]
	\arrow[from=1-3, to=1-2]
	\arrow[from=1-3, to=2-3]
	\arrow[from=1-4, to=1-3]
\end{tikzcd}\]
where $\overline{\text{E}}$ is the cofiber of the unit map $\mathbb{S} \to \text{E}$. For any motivic spectrum $\mathrm{X}$, we can apply the functor $\pi_{*,*}^{\mathbb{F}_q}(\mathrm{X} \otimes-)$. This yields the \textit{\textup{E}-based motivic Adams spectral sequence for $X$}. By construction, the $\mathrm{E}_1$-page of this spectral sequence has signature
\[\mathrm{E}^{s,f,w}_1 = \pi_{s+f, w}^{\mathbb{F}_q}(\text{E} \otimes \overline{\text{E}}^{\otimes f} \otimes \mathrm{X}) \implies \pi_{s, f}^{\mathbb{F}_q}(\mathrm{X}_\text{E}), \quad d_r:\mathrm{E}_r^{s,f,w} \to \mathrm{E}_r^{s-1, f+r, w}\]
where $\mathrm{X}_\text{E}$ denotes the E-nilpotent completion of X \cite{Bousfield-localization}. Convergence is in general not immediate, but we will only be interested in particular cases that are well-understood.

For $\text{E}=\text{kq}$, the resulting motivic Adams spectral sequence is called the kq-\textit{resolution}. This has signature
\[
\mathrm{E}^{s,f,w}_1 = \pi_{s+f, w}^{\mathbb{F}_q}(\text{kq} \otimes \overline{\text{kq}}^{\otimes f} \otimes \mathrm{X}) \implies \pi_{s, f}^{\mathbb{F}_q}(\mathrm{X}).
\]
Convergence of the kq-resolution was proven for $X=\mathbb{S}$ in \cite[Theorem 2.1]{CQ21} over any base field of characteristic different than 2. Culver--Quigley use the $\mathbb{C}$-motivic kq-resolution for the sphere to compute the $v_1$-periodic stable stems. We computed the $\text{E}_1$-page of the kq-resolution in $\mathbb{R}$-motivic homotopy theory in \cite{Realkqcoop}. The primary aim of this paper is to compute the $\mathrm{E}_1$-page of the kq-resolution in $\mathbb{F}_q$-motivic homotopy theory. 

For $\text{E} = \text{H}\mathbb{F}_2$, the resulting motivic Adams spectral sequence will be referred to as the $\textbf{mASS}^{\mathbb{F}_q}(\mathrm{X})$. In this case, since $\euscr{A}^\vee = \pi_{*,*}^{\mathbb{F}_q}(\text{H}\mathbb{F}_2 \otimes \text{H}\mathbb{F}_2)$ is flat as a module over $\mathbb{M}_2^{\mathbb{F}_q}$, we are able to pass directly to the $\mathrm{E}_2$-page. The $\textbf{mASS}^{\mathbb{F}_q}(\mathrm{X})$ has signature
\[\mathrm{E}^{s,f,w}_2 = \text{Ext}^{s,f,w}_{\euscr{A}^\vee}(\text{H}_{*,*}(\mathrm{X})) \implies \pi_{s,f}^{\mathbb{F}_q}(\mathrm{X}).\]
Convergence for this motivic Adams spectral sequence was first studied by \cite{HKO-convergencemASS,DImASS}. We will study the ring of cooperations for kq, and hence determine the $\text{E}_1$-page of the kq-resolution, by computing the $\textbf{mASS}^{\mathbb{F}_q}(\text{kq} \otimes \text{kq}).$

\begin{remark}
    Another way to extract $v_1$-periodicity is by studying the $\text{BPGL} \langle 1 \rangle$-motivic Adams spectral sequence, where $\text{BPGL} \langle 1 \rangle$ is the first truncated motivic Brown--Peterson spectrum \cite{Hu-Kriz-remarks}. The study of this spectral sequence was initiated in joint work with Petersen and Tatum \cite{MorPetTat-BPGL1}, where we computed the ring of cooperations $\pi_{*,*}^F(\mathrm{BPGL}\langle 1 \rangle \otimes \mathrm{BPGL} \langle 1 \rangle )$ at all primes and over the fields $\mathbb{C}, \mathbb{R},$ and $\mathbb{F}_q$. The truncated motivic Brown--Peterson cooperations actually offer a spectrum level decomposition, which is absent in the case of kq (see \Cref{remark:FailureForSpectrumLevelSplitting}). 
    
    Historically, at the prime 2, $v_1$-periodicity is more easily accessible by means of the $\mathrm{bo}$-resolution as opposed to the $\mathrm{BP}\langle 1 \rangle \simeq \mathrm{bu}$-resolution. However, the motivic Hopf map $\eta$, which is $v_1$-periodic, is non-nilpotent. This leads to a division in of $v_1$-periodic classes into two families: those which are $\eta$-periodic, of which we have a firm understanding \cite{AndMil17,GuiIsa16_Rmotivic_etaperiodic,Wil18_etainvertedrationals,BH21}; and those which are not $\eta$-periodic, which presumably are just as easily accessible via the $\mathrm{BPGL}\langle 1 \rangle \simeq \mathrm{kgl}$-resolution as the $\mathrm{kq}$-resolution. It will be interesting to draw comparisons between these two spectral sequences at their .
\end{remark}

\subsection{The inductive process to calculate the ring of cooperations}
\label{inductive process section}
In this section, we describe an inductive process to compute the $\text{E}_2$-page of the $\textbf{mASS}^{\mathbb{F}_q}(\text{kq} \otimes \text{kq})$. This involves a delicate analysis of many algebraic Atiyah--Hirzebruch spectral sequences. We refer the reader to \cite[Section 4]{Realkqcoop} for more details.

Observe that the $\textbf{mASS}^{\mathbb{F}_q}(\text{kq} \otimes \text{kq})$ has signature
\[\mathrm{E}_2^{s,f,w} = \text{Ext}^{s,f,w}_{\euscr{A}^\vee}(\textup{H}_{*,*}(\textup{kq} \otimes \textup{kq})) \implies \pi_{s,w}^{\mathbb{F}_q}(\textup{kq} \otimes \textup{kq}).\]
The K\"unneth spectral sequence
\[\mathrm{E}_2 = \text{Tor}^{\mathbb{M}_2^{\mathbb{F}_q}}(\text{H}_{*,*}(\text{kq}), \text{H}_{*,*}(\text{kq})) \implies \text{H}_{*,*}(\text{kq} \otimes \text{kq})\]
collapses on the $\textup{E}_2$-page \cite[Proposition 2.7]{Realkqcoop}, giving an isomorphism
\[\text{H}_{*,*}(\text{kq} \otimes \text{kq}) \cong \text{H}_{*,*}(\text{kq}) \otimes\text{H}_{*,*}(\text{kq}).\]
Combining this with the change of rings isomorphism \cite[Appendix A]{Rav86} coming from the $\euscr{A}^\vee$-comodule algebra isomorphism 
\[
\text{H}_{*,*}(\text{kq}) \cong (\euscr{A} \modmod \euscr{A}(1))^\vee
\]
and the Brown--Gitler decomposition of \cref{kq brown gitler homology}, we obtain the following.
\begin{thm}[{\cite[Theorem 4.1]{Realkqcoop}}]
\label{thm:e2 mass kq coop}
    The $\textup{E}_2$-page of the $\textup{\textbf{mASS}}^{\mathbb{F}_q}(\textup{kq} \otimes \textup{kq})$ is given by
    \[\textup{E}_2^{s,f,w} = \bigoplus_{k \geq 0}\textup{Ext}^{s,f,w}_{\euscr{A}(1)^\vee}(\Sigma^{4k, 2k}B_0^{\mathbb{F}_q}(k)).\]
\end{thm}
This decomposition, combined with the short exact sequences of \cref{prop:ses bg}, allows us to compute the $\mathrm{E}_2$-page inductively. For $k=0$, we have that $B_0^{\mathbb{F}_q}(0) =  \mathbb{M}_2^{\mathbb{F}_q},$ and so the relevant summand of the $\mathrm{E}_2$-page is $\text{Ext}_{\euscr{A}(1)^\vee}^{*,*,*}(\mathbb{M}_2^{\mathbb{F}_q}).$ We will review these Ext groups in the coming sections. 

For $k=1$, we have that $B_0^{\mathbb{F}_q}(1) \cong \mathbb{M}_2^{\mathbb{F}_q}\{1, \overline{\xi}_1, \overline{\tau}_1\}$. The relevant summand of the $\mathrm{E}_2$-page may be calculated using an \textit{algebraic Atiyah--Hirzebruch spectral sequence}, denoted $\textbf{aAHSS}(B_0^{\mathbb{F}_q}(1))$. Filtering $B_0^{\mathbb{F}_q}(1)$ by topological degree and letting $\text{F}_iB_0^{\mathbb{F}_q}(1)$ denote the subspace spanned by generators of degree $\leq i$ gives a finite filtration of the form
\[\begin{tikzcd}
	0 & {\text{F}_0B_0^{\mathbb{F}_q}(1)} & {\text{F}_1B_0^{\mathbb{F}_q}(1)} & {\text{F}_2B_0^{\mathbb{F}_q}(1)} & {\text{F}_3B_0^{\mathbb{F}_q}(1) = B_0^F(1)} \\
	& {\mathbb{M}_2^{\mathbb{F}_q}\{[1]\}} & 0 & {\mathbb{M}_2^{\mathbb{F}_q}\{[\overline{\xi}_1]\}} & {\mathbb{M}_2^{\mathbb{F}_q}\{[\overline{\tau}_1]\}}
	\arrow[from=1-1, to=1-2]
	\arrow[from=1-2, to=1-3]
	\arrow[from=1-2, to=2-2]
	\arrow[from=1-3, to=1-4]
	\arrow[from=1-3, to=2-3]
	\arrow[from=1-4, to=1-5]
	\arrow[from=1-4, to=2-4]
	\arrow[from=1-5, to=2-5]
\end{tikzcd}\]
The $\textbf{aAHSS}(B_0^{\mathbb{F}_q}(1))$ comes from applying the functor $\text{Ext}^{*,*,*}_{\euscr{A}(1)^\vee}(-)$ to this filtration. This spectral sequence has signature
\[\mathrm{E}_1^{s,f,w,a} = \text{Ext}^{s,f,w}_{\euscr{A}(1)^\vee}(\mathbb{M}_2^{\mathbb{F}_q}) \otimes \mathbb{M}^{\mathbb{F}_q}_2\{[1], [\overline{\xi}_1], [\overline{\tau}_1]\}\implies \text{Ext}^{s,f,w}_{\euscr{A}(1)^\vee}(B_0^{\mathbb{F}_q}(1)),\]
where $a$ denotes the Atiyah--Hirzebruch filtration degree.
We let $\alpha[i]$ denote the copy of the class $\alpha \in \text{Ext}^{*,*,*}_{\euscr{A}(1)^\vee}(\mathbb{M}_2^{\mathbb{F}_q})$ in Atiyah--Hirzebruch filtration $i$. The differentials in the $\textbf{aAHSS}(B_0^{\mathbb{F}_q}(1))$ are of the form
\[d_r:\mathrm{E}_r^{s,f,w,a} \to \mathrm{E}_r^{s-1, f+1,w, a-r}.\]
Thus the $d_r$-differential lowers Atiyah--Hirzebruch filtration by $r$. In particular, the structure of $B_0^{\mathbb{F}_q}(1)$ ensures that each differential only takes nonzero value on one Atiyah--Hirzebruch filtration piece. By inspection of the length of the cellular filtration on $B_0^{\mathbb{F}_q}(1)$, we have the following.
\begin{proposition}
\label{aAHSS convergenece general}
    In the $\textup{\textbf{aAHSS}}(B_0^{\mathbb{F}_q}(1))$, we have that $\textup{E}_4=\textup{E}_\infty$.
\end{proposition}
The differentials in the $\textbf{aAHSS}(B_0^{\mathbb{F}_q}(1))$ are determined by the cobar complex; see \cite[Section 4.2]{Realkqcoop}. We recall the behavior of the $d_1$ and $d_2$-differentials.
\begin{proposition}[{\cite[Proposition 4.2, Proposition 4.3]{Realkqcoop}}]
\label{aAHSS d1}
    In the $\textup{\textbf{aAHSS}}(B_0^{\mathbb{F}_q}(1))$, the $d_1$-differential is determined by
    \[d_1(\alpha[3]) = h_0\alpha[2],\]
    and the $d_2$-differential is determined by
    \[d_2(\alpha[2]) = h_1\alpha[0].\]
\end{proposition}
Since the differentials in the $\textbf{aAHSS}(B^{\mathbb{F}_q}_0(1))$ are linear over $\text{Ext}_{\euscr{A}(1)^\vee}^{*,*,*}(\mathbb{M}_2^{\mathbb{F}_q})$, these formulae completely determine the $d_1$ and $d_2$-differentials. By inspection, the only possible nonzero $d_3$-differential is between Atiyah--Hirzebruch filtrations 3 and 0, and will be given by the Massey product
\[d_3(\alpha[3]) = \langle\alpha, h_0, h_1 \rangle[0].\]
We will see that the nontriviality of this differential depends on whether or not $q \equiv 1 \, (4)$ or $q \equiv 3 \, (4)$.

With the computation of $\text{Ext}_{\euscr{A}(1)^\vee}^{*,*,*}(B_0^{\mathbb{F}_q}(1))$, we may proceed to calculate the rest of the $\mathrm{E}_2$-page of the $\textbf{mASS}^{\mathbb{F}_q}(\text{kq} \otimes \text{kq})$ by using the short exact sequences of \cref{prop:ses bg}. Since we are primarily interested in using the kq-resolution to determine the $v_1$-periodic elements in $\pi_{*,*}^{\mathbb{F}_q}(\mathbb{S})$, we will compute the $\mathrm{E}_2$-page modulo $v_1$-torsion, which makes the presentation of our computations significantly more concise.

Following these calculations, we will be prepared to compute the ring of cooperations. First, using \cref{eq:MotivicCohomology}, we will compute the $\textbf{mASS}^{\mathbb{F}_q}(\text{H}\mathbb{Z})$. Using the quotient map $\text{kq} \to \text{H}\mathbb{Z}$, we will lift the differentials to the $\textbf{mASS}^{\mathbb{F}_q}(\text{kq})$, where we are aided by Friedlander's calculations of the hermitian K-theory of finite fields. Then, using the kq-module structure on $\text{kq} \otimes \text{kq}$, we will be able to determine the differentials in the $\textbf{mASS}^{\mathbb{F}_q}(\text{kq} \otimes \text{kq}).$

\begin{remark}
    It is worthwhile to note that, unlike in the classical, $\mathbb{C}$-motivic, and $\mathbb{R}$-motivic cases, the $\textbf{mASS}^{\mathbb{F}_q}(\text{H}\mathbb{Z})$ has room for potential differentials. In fact, by the isomorphism
    \[\pi_{s,w}^{\mathbb{F}_q}(\text{H}\mathbb{Z}) \cong \text{H}^{-s, -w}(\mathbb{F}_q; \mathbb{Z})\]
    and the content of \cref{eq:MotivicCohomology}, we will see that there must be infinitely many differentials. Lifting along the quotient map $\text{kq} \to \text{H}\mathbb{Z}$ shows that there are infinitely many differentials in the $\textbf{mASS}^{\mathbb{F}_q}(\text{kq})$, and these also persist to the $\textbf{mASS}^{\mathbb{F}_q}(\text{kq} \otimes \text{kq}).$
\end{remark}

\section{Trivial Bockstein action}
\label{section:trivial}
In this section, we compute the ring of cooperations $\pi_{*,*}^{\mathbb{F}_q}(\text{kq} \otimes \text{kq})$ in the case where there is a trivial Bockstein action on $\mathbb{F}_q$, that is, when $q \equiv 1 \, (4)$. We first describe the $\textbf{mASS}^{\mathbb{F}_q}(\text{H}\mathbb{Z})$, then show that the differentials lift to the $\textbf{mASS}^{\mathbb{F}_q}(\text{kq})$, completely determining the behavior of the spectral sequence. After calculating the structure of the $\mathrm{E}_2$-page of the $\textbf{mASS}^{\mathbb{F}_q}(\text{kq} \otimes \text{kq})$ as a module over the $\mathrm{E}_2$-page of the $\textbf{mASS}^{\mathbb{F}_q}(\text{kq})$, we are able to lift differentials again, determining the spectral sequence and the ring of cooperations. 

\subsection{Integral homology}
We begin by computing the $\textbf{mASS}^{\mathbb{F}_q}(\text{H}\mathbb{Z})$. This spectral sequence has signature
\[
\mathrm{E}_2^{s,f,w} = \text{Ext}^{s,f,w}_{\euscr{A}^\vee}(\text{H}_{*,*}(\text{H}\mathbb{Z})) \implies \pi_{s,w}^{\mathbb{F}_q}(\text{H}\mathbb{Z}).
\]
Recall that there is an isomorphism of $\euscr{A}^\vee$-comodule algebras $\text{H}_{*,*}(\text{H}\mathbb{Z}) \cong (\euscr{A} \modmod \euscr{A}(0))^\vee.$
This allows us to rewrite the $\mathrm{E}_2$-page of the $\textbf{mASS}^{\mathbb{F}_q}(\mathrm{H}\mathbb{Z})$ using a change of rings isomorphism as
\[\text{Ext}_{\euscr{A}^\vee}^{s,f,w}((\euscr{A} \modmod \euscr{A}(0))^\vee) \cong \text{Ext}^{s,f,w}_{\euscr{A}(0)^\vee}(\mathbb{M}_2^{\mathbb{F}_q}).\]
This Ext group was calculated by Kylling to be \cite[Theorem 4.1.2]{Kylingkqfinite}
\[\text{Ext}^{*,*,*}_{\euscr{A}(0)^\vee}(\mathbb{M}_2^{\mathbb{F}_q}) \cong \mathbb{F}_2[u,\tau, h_0]/(u^2),\]
where $|h_0| = (1,0,0).$ 

We depict this $\mathrm{E}_2$-page in \Cref{f5_Ext_A0}. Note that all charts are depicted in $(s,f)$-grading, and that motivic weight is suppressed.

\begin{figure}
    \centering    
    \includegraphics[scale=.75]{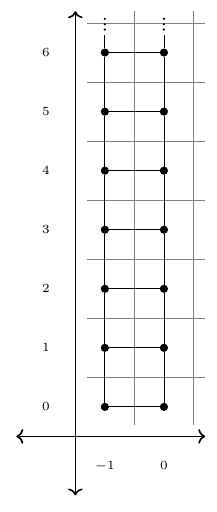}
    \caption{$\text{Ext}^{*,*,*}_{\euscr{A}(0)^\vee}(\mathbb{M}_2^{\mathbb{F}_q})$ for $q \equiv 1 \, (4)$. 
    }
    \label{f5_Ext_A0}
\end{figure}

We now determine all of the differentials in the spectral sequence.
\begin{lemma}[{\cite[Lemma 4.2.1]{Kylingkqfinite}}] 
\label{F5 difs mass HZ}
    Let $q \equiv 1 \, (4)$. The differentials in the $\textup{\textbf{mASS}}^{\mathbb{F}_q}(\textup{H}\mathbb{Z})$ are determined by
    \[d_{\nu_2(q-1)+\nu_2(i)}(\tau^i) = u\tau^{i-1}h_0^{\nu_2(q-1)+\nu_2(i)}.\]
\end{lemma}

\begin{proof}
    Since $\pi_{0,0}^{\mathbb{F}_q}(\text{H}\mathbb{Z}) \cong \mathbb{Z}_2$, any element in $\mathrm{E}_2^{0,f,0}$ must be a permanent cycle for all $f \geq 0.$ This implies that $1$ and all powers of $h_0$ are permanent cycles. Thus, if a class $x$ supports a differential, then $x$ must be divisible by $\tau$. Suppose that $y$ is not divisible by $\tau$ and that $x=\tau^ky$ for some $k \geq 1$. Then by the Liebniz rule, for any $r \geq 0$ we have
    \[d_r(x) = d_r(y)\tau^k + yd_r(\tau^k) = yd_r(\tau^k).\]
    Thus all differentials are uniquely determined by their value on $\tau$. Moreover, $\pi_{0, -n}^{\mathbb{F}_q}(\text{H}\mathbb{Z})=0$ for any $n \geq 1$, so there can no $\tau$-divisibility in stem $0$ on the $\mathrm{E}_\infty$-page. The particular values for the differentials are then a simple consequences of \eqref{eq:MotivicCohomology} combined with the fact that Adams differentials preserve motivic weight. 
\end{proof}

To illustrate the behavior of the differentials in the $\textbf{mASS}^{\mathbb{F}_q}(\text{H}\mathbb{Z})$, we give an example.

\begin{example}
\label{ex:f5-mASS-HZ}
    Let $q=5$. We have that $\pi_{-1, -1}^{\mathbb{F}_5}(\text{H}\mathbb{Z}) \cong \mathbb{Z}/4$, which implies that the class $u h_0^2$ is the class in stem $-1$ and weight $-1$ of lowest Adams filtration and greatest motivic weight which must be the target of a differential. Since $\tau$ must die, this forces the differential $d_{2}(\tau) = u  h_0^2$. The Liebniz rule then determines all $d_2$-differentials on odd powers of $\tau$. Similarly, we have $\pi_{-1, -2}^{\mathbb{F}_5}(\text{H}\mathbb{Z}) \cong \mathbb{Z}/8$, so we must have that $d_{3}(\tau^2) = u\tau h_0^3$, and the Liebniz rule determines all $d_3$-differentials on $\tau^{2n}$ for $n$ odd.
    
    We depict the $\mathrm{E}_\infty$-page in the case of $q=5$ in \Cref{fig:f5-Einfty-HZ}. An empty $\circ$ denotes $\mathbb{F}_2$. A class labeled with $n$ indicates the first nonzero power of $\tau$ which exists in that particular bidegree $(s,f)$. An empty class with $k$ circles around it indicates $\tau^{2^k}$-periodicity. For example, in bidegree $(-1, 2)$ we have $\mathbb{F}_2\{u\tau h_0^2\}[\tau^2],$ and in bidegree $(-1, 3)$ we have $\mathbb{F}_2\{u\tau^2h_0^3\}[\tau^4].$
\end{example}

\begin{figure}
    \centering
    \includegraphics[scale=.75]{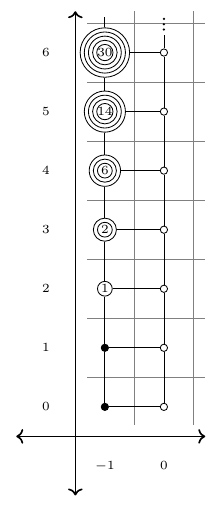}
    \caption{The $\mathrm{E}_\infty$-page of the $\textbf{mASS}^{\mathbb{F}_q}(\text{H}\mathbb{Z})$ for $q=5$. 
    }
    \label{fig:f5-Einfty-HZ}
\end{figure}
\subsection{Hermitian K-theory}
We now compute the $\textbf{mASS}^{\mathbb{F}_q}(\text{kq}).$ This spectral sequence has signature
\[\mathrm{E}^{s,f,w}_2 = \text{Ext}^{s,f,w}_{\euscr{A}^\vee}(\text{H}_{*,*}(\text{kq})) \implies \pi_{s,w}^{\mathbb{F}_q}(\text{kq}).\]
Recall that there is an isomorphism of $\euscr{A}^\vee$-comodule algebras $\text{H}_{*,*}(\text{kq}) \cong (\euscr{A} \modmod \euscr{A}(1))^\vee.$
This allows us to rewrite the $\mathrm{E}_2$-page of the $\textbf{mASS}^{\mathbb{F}_q}(\mathrm{kq})$ using a change of rings isomorphism as
\[\text{Ext}_{\euscr{A}^\vee}^{s,f,w}((\euscr{A}\modmod \euscr{A}(1))^\vee) \cong \text{Ext}_{\euscr{A}(1)^\vee}^{s,f,w}(\mathbb{M}_2^{\mathbb{F}_q}).\]
This Ext group was calculated by Kylling to be \cite[Theorem 4.1.2]{Kylingkqfinite}:
\begin{equation}
\label{Ext_A(1)_m2_easy}
\text{Ext}^{*,*,*}_{\euscr{A}(1)^\vee}(\mathbb{M}_2^{\mathbb{F}_q}) = 
\frac{\mathbb{F}_2[u, \tau,  h_0, h_1, a, b]}{(u^2,h_0h_1, \tau h_1^3, h_1 a, a^2=h_0^2b)},
\end{equation}
where $|h_0| = (1,0,0)$, $|h_1|=(1,1,1)$, $|a| = (4,3, 2)$, and $|b| = (8,4,4)$. We note that this bears similarity to the $\mathbb{C}$-motivic analogue \cite{Hil11}:
\[
\text{Ext}^{*,*,*}_{\euscr{A}(1)^\vee}(\mathbb{M}_2^{\mathbb{C}}) = 
\frac{\mathbb{F}_2[ \tau,  h_0, h_1, a, b]}{(h_0h_1, \tau h_1^3, h_1 a, a^2=h_0^2b)},
\]
with generators in the same tridegree as in \cref{Ext_A(1)_m2_easy}. Indeed, there is an abstract isomorphism
\begin{equation}
\label{eq:ExtA1F5}
\text{Ext}^{*,*,*}_{\euscr{A}(1)^\vee}(\mathbb{M}_2^{\mathbb{F}_q}) \cong \text{Ext}_{\euscr{A}(1)^\vee}^{*,*,*}(\mathbb{M}_2^\mathbb{C})\otimes \mathbb{M}_2^{\mathbb{F}_q}.
\end{equation}
We depict this Ext group in \Cref{F5_Ext_A1}. 

\begin{figure}[h]
    \centering
    \includegraphics[scale=.75]{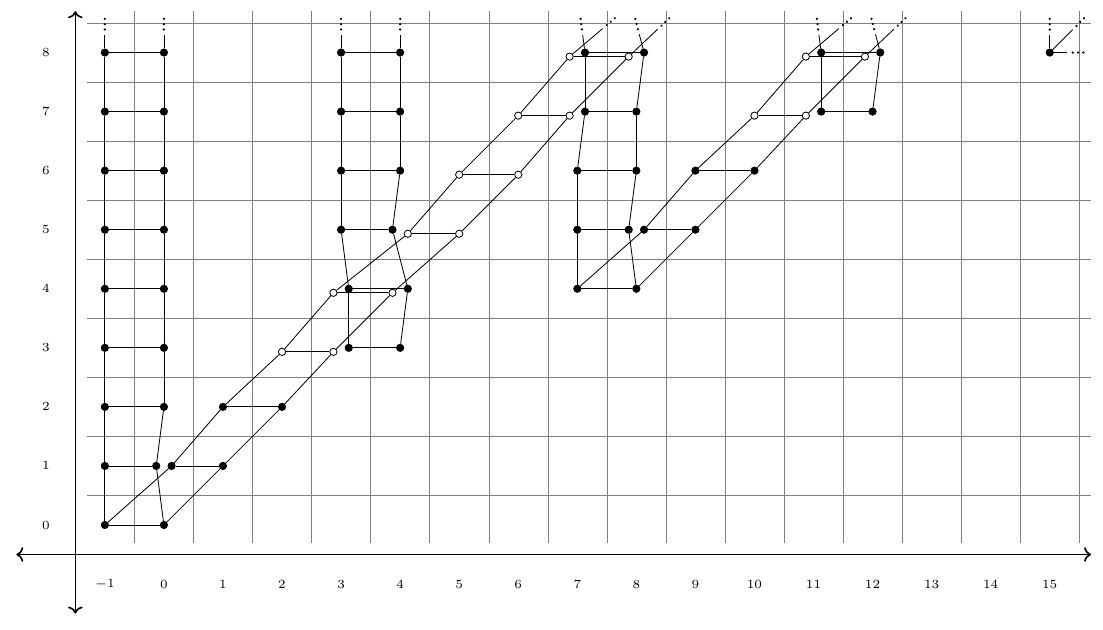}
    \caption{$\text{Ext}^{*,*,*}_{\euscr{A}(1)^\vee}(\mathbb{M}_2^{\mathbb{F}_q})$ for $q \equiv 1 \, (4)$.
    A $\bullet$ denotes $\mathbb{F}_2[ \tau]$, and a $\circ$ denotes $\mathbb{F}_2$. A vertical line represents $h_0$-multiplication, a line of slope 1 represents $h_1$-multiplication, and a vertical line represents $u$-multiplication.}
    \label{F5_Ext_A1}
\end{figure}

We will now show that we can lift the differentials from the $\textbf{mASS}^{\mathbb{F}_q}(\text{H}\mathbb{Z})$ to the $\textbf{mASS}^{\mathbb{F}_q}(\text{kq})$; see also \cite[Section 4.2]{Kylingkqfinite}.

\begin{proposition}
\label{prop:EasyHZDifsLiftTokq}
    Let $q \equiv 1 \, (4)$. The differentials in the $\textup{\textbf{mASS}}^{\mathbb{F}_q}(\textup{kq})$ are determined by the differentials in the $\textup{\textbf{mASS}}^{\mathbb{F}_q}(\textup{H}\mathbb{Z}).$   
\end{proposition}

\begin{proof}
    We first determine the possible Adams differentials on the generators of the $\mathrm{E}_2$-page. We see that $u$ cannot support a differential, as Adams differentials decrease stem degree by 1 and there are no classes in stem -2. Further, Adams differentials preserve weight. Since there are no classes in weight 0 in stem -1, as they are all divisible by $u$, we cannot have any differentials on $1$ or any power of $h_0$. If $h_1$ were to support a differential, then degree considerations force it to be of the form $d_r(h_1) = \tau h_0^{r+1}$. However, this would imply that $d_r(h_1h_0) = d_r(h_1)h_0 + h_1d_r(h_0) = \tau h_0^{r+2}$, which is impossible since $h_1h_0=0$. Thus $h_1$ does not support a differential. Finally, the only chance for $b$ to support a differential is $d_3(b) = h_1^7$, but since the weight of $b$ is 4 and the weight of $h_1^7$ is 7, this cannot happen. So, we see that if a class $x$ supports a differential, it must be $\tau$-divisible.
    
    Recall the quotient map $\text{kq} \to \text{H}\mathbb{Z}$ of \eqref{eq:kqToHZ}. This induces a map of spectral sequences $\textbf{mASS}^{\mathbb{F}_q}(\text{kq}) \to \textbf{mASS}^{\mathbb{F}_q}(\text{H}\mathbb{Z})$, which is realized on $\mathrm{E}_2$-pages as the evident quotient map
    \[
    \text{Ext}^{s,f,w}_{\euscr{A}(1)^\vee}(\mathbb{M}_2^{\mathbb{F}_q}) = 
    \frac{\mathbb{F}_2[u,\tau, h_0, h_1, a, b]}{(u^2,h_0h_1, \tau h_1^3, h_1 a, a^2=h_0^2b)} \to \mathbb{F}_2[u,\tau,  h_0]/(u^2) = \text{Ext}_{\euscr{A}(0)^\vee}^{s,f,w}(\mathbb{M}_2^{\mathbb{F}_q}),
    \]
    where $h_1, a, b \mapsto0$.
    In particular, this implies that the differentials on powers of $\tau$ determined by \Cref{F5 difs mass HZ} lift to the $\textbf{mASS}^{\mathbb{F}_q}(\text{kq}).$ Since $u, h_0, h_1, a,$ and $b$ are permanent cycles, the Liebniz rule determines all differentials in the spectral sequence, finishing the proof.
\end{proof}
Running the $\textbf{mASS}^{\mathbb{F}_q}(\text{kq})$, we see that
\[\pi_{s,w}^{\mathbb{F}_q}(\text{kq}) \cong\left\{\begin{array}{rl}
    \text{KO}_{s-2w}(\mathbb{F}_q) & w \equiv 0 \, (4) \\
    \text{KSp}_{s-2w}(\mathbb{F}_q) & w \equiv 2 \ (4) 
\end{array} \right.\]
for $s \geq 0$, agreeing with Friedlander's calculations given in \Cref{table:fried}. 

\begin{remark}
\label{rem:f5NegativeStemkq}
    By lifting differentials from the $\textbf{mASS}^{\mathbb{F}_q}(\text{H}\mathbb{Z})$, we are able to more clearly identify the values of $\pi_{s,w}^{\mathbb{F}_q}(\text{kq})$ for $s \leq0$. The quotient map $\text{kq} \to \text{H}\mathbb{Z}$ induces an isomorphism of motivic Adams spectral sequence $\mathrm{E}_2$-pages in stems $s \leq 0$, hence the differentials in the $\textbf{mASS}^{\mathbb{F}_q}(\text{kq})$ are isomorphic to the differentials in the $\textbf{mASS}^{\mathbb{F}_q}(\text{H}\mathbb{Z})$ in this range. This gives an isomorphism $\pi_{s,w}^{\mathbb{F}_q}(\text{kq}) \cong \pi_{s,w}^{\mathbb{F}_q}(\text{H}\mathbb{Z})$ for $s \leq 0$. More generally, we can express the homotopy of kq in terms of $\text{H}\mathbb{Z}$:
    \[
    \pi_{*,*}^F(\text{kq}) \cong \frac{(\pi_{*,*}^F(\text{H}\mathbb{Z}))[\eta, \tau^n\eta, \tau^n\eta^2, \alpha, \beta: n \geq 0]}{(2\eta, \eta\alpha, \tau^n\eta\alpha, \tau^n\eta^2\alpha, \alpha^2-4\beta)}, 
    \]
    where $\eta$ is detected by $h_1$, $\alpha$ is detected by $a$, and $\beta$ is detected by $b$. Note that the additional generators $\tau^n\eta, \tau^n\eta^2$ stem from the fact that $\tau$ dies in the $\textbf{mASS}^{\mathbb{F}_q}(\text{H}\mathbb{Z})$ while $\tau \eta$ survives in the $\textbf{mASS}^{\mathbb{F}_q}(\text{kq})$.
\end{remark}

\subsection{Brown--Gitler comodules}
We begin the inductive of computing the $\mathrm{E}_2$-page of the $\textbf{mASS}^{\mathbb{F}_q}(\text{kq} \otimes \text{kq})$. To start, we will compute $\text{Ext}_{\euscr{A}(1)^\vee}^{s,f,w}(B_0^{\mathbb{F}_q}(1))$ by the $\textbf{aAHSS}^{\mathbb{F}_q}(B_0^{\mathbb{F}_q}(1))$. This spectral sequence has signature
\[\mathrm{E}_1^{s,f,w,a} = \text{Ext}^{s,f,w}_{\euscr{A}(1)}(\mathbb{M}_2^{\mathbb{F}_q}) \otimes \mathbb{M}^{\mathbb{F}_q}_2\{[1], [\overline{\xi}_1], [\overline{\tau}_1]\}\implies \text{Ext}^{s,f,w}_{\euscr{A}(1)}(B_0^{\mathbb{F}_q}(1))\]
with differentials taking the form
\[d_r:\mathrm{E}_r^{s,f,w,a} \to \mathrm{E}_r^{s-1, f+1,w, a-r}.\]
We have discussed the differentials in this spectral sequence in \Cref{inductive process section}. In particular, for any $\alpha \in \text{Ext}_{\euscr{A}(1)^\vee}^{*,*,*}(\mathbb{M}_2^{\mathbb{F}_q})$, we have that
\[d_1(\alpha[3]) = h_0\alpha[2],\]
and
\[d_2(\alpha[2]) = h_1\alpha[0].\]
There is a potential third differential between Atiyah--Hirzebruch filtrations 3 and 0, but we see that for degree reasons this is not possible. Since the differentials are module maps over $\text{Ext}_{\euscr{A}(1)^\vee}^{*,*,*}(\mathbb{M}_2^{\mathbb{F}_q})$, this determines the spectral sequence. Hidden extensions lift in the same way as in the $\mathbb{C}$-motivic case; see \cite[Section 3.4]{CQ21}. We depict the $\mathrm{E}_\infty$-page in \Cref{fig:F5_aAHSS_Ext_B01_Einfty}. 

\begin{figure}[h]
    \centering
    \includegraphics[scale=.75]{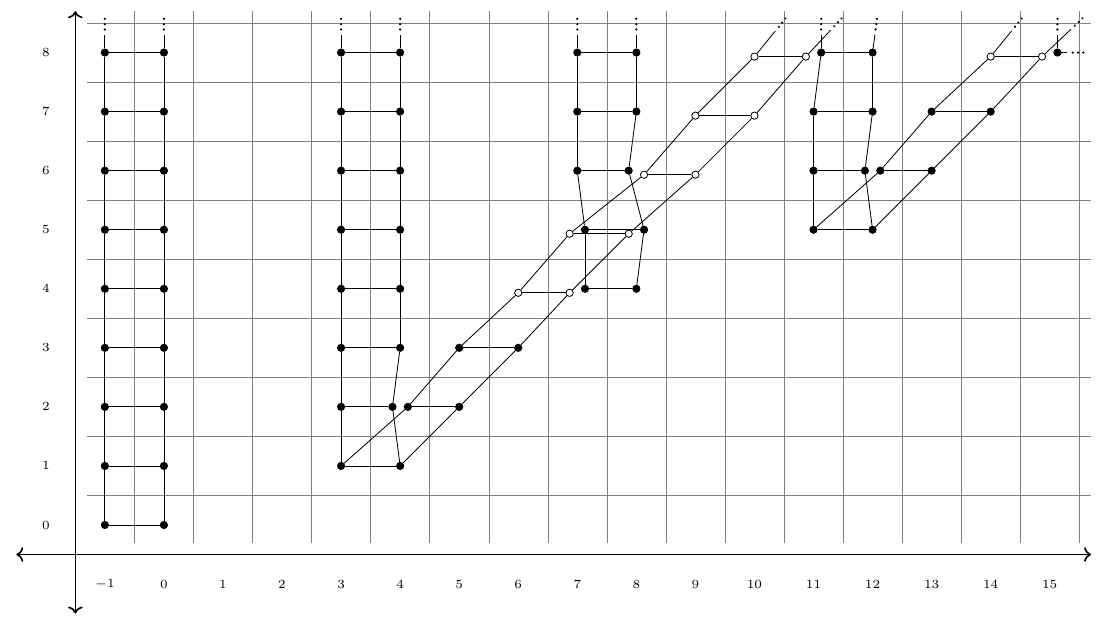}
    \caption{The $\mathrm{E}_\infty$-page of the $\textbf{aAHSS}(B_0^{\mathbb{F}_q}(1))$ for $q \equiv 1 \, (4)$.}
    \label{fig:F5_aAHSS_Ext_B01_Einfty}
\end{figure}

The equivalence of motivic spectra $\text{ksp} \simeq \text{H}\mathbb{Z}_1^{\mathbb{F}_q} \otimes \text{kq}$ from \cref{eq:kspSplitting}
implies that the $\textup{E}_2$-page of the $\textbf{mASS}^{\mathbb{F}_q}(\text{ksp})$ is given by $\text{Ext}_{\euscr{A}(1)^\vee}^{*,*,*}(B_0^{\mathbb{F}_q}(1)).$ 
The K\"unneth spectral sequence computing $\text{H}_{*,*}(\text{ksp})$ collapses, and so the change of rings isomorphism implies that the $\mathrm{E}_2$-page of the $\textbf{mASS}^{\mathbb{F}_q}(\text{ksp})$ is isomorphic to $\text{Ext}_{\euscr{A}(1)^\vee}^{*,*,*}(B_0^{\mathbb{F}_q}(1)).$ However, there are differentials in this spectral sequence.

\begin{lemma}
\label{lem:kspmASSDifsEasy}
    Let $q \equiv 1 \, (4)$. In the $\textup{\textbf{mASS}}^{\mathbb{F}_q}(\textup{ksp})$, the differentials are determined by the differentials in the $\textup{\textbf{mASS}}^{\mathbb{F}_q}(\textup{kq})$ and the $\textup{\textbf{mASS}}^{\mathbb{F}_q}(\textup{H}\mathbb{Z}).$
\end{lemma}

\begin{proof}
    Recall the cofiber sequence of \cref{cofib:kqTokspToHZ}
    \[\Sigma^{4,2}\text{kq} \to \text{ksp} \to \text{H}\mathbb{Z}.\]
    By applying $\text{Ext}_{\euscr{A}(1)^\vee}^{*,*,*}(\text{H}_{*,*}(-))$ to this cofiber sequence, we get a long exact sequence in Ext groups. Notice that these Ext groups are the $\mathrm{E}_2$-pages for the respective motivic Adams spectral sequences. For degree reasons, the connecting homomorphism
    \[\delta:\text{Ext}_{\euscr{A}^\vee}^{s,f,w}(\text{H}_{*,*}(\text{H}\mathbb{Z})) \to \text{Ext}_{\euscr{A}^\vee}^{s-1, f+1, w}(\text{H}_{*,*}(\Sigma^{4,2}\text{kq}))\]
    is trivial. Thus, after a few change of rings isomorphisms, we have a decomposition of the $\mathrm{E}_2$-page of the $\textbf{mASS}^{\mathbb{F}_q}(\text{ksp})$:
    \[\text{Ext}_{\euscr{A}(1)^\vee}^{s,f,w}(B_0^{\mathbb{F}_q}(1)) \cong \Sigma^{4,2}\text{Ext}_{\euscr{A}(1)^\vee}^{s,f,w}(\mathbb{M}_2^{\mathbb{F}_q})\langle 1 \rangle \oplus \text{Ext}_{\euscr{A}(0)^\vee}^{s,f,w}(\mathbb{M}_2^{\mathbb{F}_q}),\] 
    where $\langle-\rangle$ denotes a shift in the Adams filtration degree. 
    
    The projection $\text{ksp} \to \text{H}\mathbb{Z}$ induces evident quotient map on $\mathrm{E}_2$-pages which lift the differentials from \Cref{F5 difs mass HZ}. The inclusion $\Sigma^{4,2}\text{kq} \to \text{ksp}$ induces the evident inclusion on $\mathrm{E}_2$-pages, lifting the differentials from \Cref{prop:EasyHZDifsLiftTokq} to the $\textbf{mASS}^{\mathbb{F}_q}(\text{ksp})$ to this summand. For degree reasons there are no other differentials possible, finishing the proof.
\end{proof}

By running the $\textbf{mASS}^{\mathbb{F}_q}(\text{ksp})$ and comparing with \Cref{table:fried}, we recover the expected isomorphism for $s \geq 4$
\[\pi_{s,w}^{\mathbb{F}_q}(\text{ksp}) \cong\left\{\begin{array}{rl}
    \text{KSp}_{s-2w}(\mathbb{F}_q) & w \equiv 0 \, (4); \\
    \text{KO}_{s-2w}(\mathbb{F}_q) & w \equiv 2 \ (4). 
\end{array} \right.\]

The following is immediate from our description.
\begin{lemma}
\label{lemma:easyExtB01}
    Let $q \equiv 1 \, (4)$. There is an isomorphism of $\textup{Ext}_{\euscr{A}(1)^\vee}^{*,*,*}(\mathbb{M}_2^{\mathbb{F}_q})$-modules
    \[\textup{Ext}^{s,f,w}_{\euscr{A}(1)^\vee}(B_0^{\mathbb{F}_q}(1)) \cong \textup{Ext}_{\euscr{A}(1)^\vee}^{s,f,w}(B_0^\mathbb{C}(1)) \otimes \mathbb{M}_2^{\mathbb{F}_q}.\]
\end{lemma}

\begin{proposition}
\label{prop:easyExtB01powers}
    Let $q \equiv 1 \, (4)$. For all $i \geq 0$, there is an isomorphism of $\textup{Ext}_{\euscr{A}(1)^\vee}^{*,*,*}(\mathbb{M}_2^{\mathbb{F}_q})$-modules
    \[\frac{\textup{Ext}^{s,f,w}_{\euscr{A}(1)^\vee}(B_0^{\mathbb{F}_q}(1)^{\otimes i})}{v_1\textup{-torsion}} \cong \frac{\textup{Ext}_{\euscr{A}(1)^\vee}^{s,f,w}(B_0^{\mathbb{C}}(1)^{\otimes i})}{v_1\textup{-torsion}} \otimes  \mathbb{M}_2^{\mathbb{F}_q}.\]
\end{proposition}

\begin{proof}
    The case $i=0$ was shown in \cref{eq:ExtA1F5}, and the case $i=1$ was shown in \Cref{lemma:easyExtB01}. Assume the result holds for all $n \leq i$. We may calculate the Ext group in question by another algebraic Atiyah--Hirzebruch spectral sequence. For the rest of this proof, we implicitly compute modulo $v_1$-torsion. By applying the functor $\text{Ext}^{*,*,*}_{\euscr{A}(1)^\vee}(B_0^{\mathbb{F}_q}(1)^{\otimes i} \otimes -)$ to the cellular filtration of $B_0^{\mathbb{F}_q}(1)$, we obtain a spectral sequence with signature
    \[\mathrm{E}_1^{s,f,w,a} = \text{Ext}_{\euscr{A}(1)^\vee}^{s,f,w}(B_0^{\mathbb{F}_q}(1)^{\otimes i}) \otimes \mathbb{M}_2^{\mathbb{F}_q}\{[1], [\overline{\xi}_1], [\overline{\tau}_1]\} \implies \text{Ext}^{s,f,w}_{\euscr{A}(1)^\vee}(B_0^{\mathbb{F}_q}(1)^{\otimes i+1}).\]
    By induction, we can rewrite the $\mathrm{E}_1$-page as
    \[\left(\text{Ext}^{s,f,w}_{\euscr{A}(1)^\vee}(B_0^{\mathbb{C}}(1)^{\otimes i}) \otimes \mathbb{M}_2^\mathbb{C}\{1, [\overline{\xi}_1], [\overline{\tau}_1]\}\right) \otimes \mathbb{M}_2^{\mathbb{F}_q}.\]
    As our spectral sequence is induced by the same filtration as the $\textbf{aAHSS}(B_0^{\mathbb{F}_q}(1))$, the differentials are determined by the same formulae. For degree reasons, we see that there are no $d_3$-differentials. Hidden extensions follow as in the classical case. Since the differentials are the same as in the $\mathbb{C}$-motivic case by \cite[Lemma 3.36]{CQ21}, we see that the abutment is isomorphic to 
    \[\text{Ext}_{\euscr{A}(1)^\vee}^{s,f,w}(B_0^{\mathbb{C}}(1)^{\otimes i+1}) \otimes \mathbb{M}_2^{\mathbb{F}_q},\]
    which is what we wanted to show.
\end{proof}

\begin{thm}
\label{thm:easyExtB0k}
    Let $q \equiv 1 \, (4)$. For all $k \geq 0$, there is an isomorphism
    \[\frac{\textup{Ext}^{s,f,w}_{\euscr{A}(1)^\vee}(B_0^{\mathbb{F}_q}(k))}{v_1\textup{-torsion}} \cong \frac{\textup{Ext}^{s,f,w}_{\euscr{A}(1)^\vee}(B_0^{\mathbb{C}}(k))}{v_1\textup{-torsion}} \otimes\mathbb{M}_2^{\mathbb{F}_q} .\]
\end{thm}

\begin{proof}
    We use the short exact sequences of $\euscr{A}(1)^\vee$-comodules described in \Cref{prop:ses bg}. For convenience, all Ext groups are implictly computed modulo $v_1$-torsion in this proof. The case of $k=1$ was shown in \Cref{lemma:f3HZdifs}. Now, suppose that the theorem is true for all $i <k$.

    Suppose that $k$ is even. There is a short exact sequence of Brown--Gitler comodules
    \[0 \to \Sigma^{2k, k}B_0^{\mathbb{F}_q}(\tfrac{k}{2}) \to B_0^{\mathbb{F}_q}(k) \to B_1^{\mathbb{F}_q}(\tfrac{k}{2}-1) \otimes (\euscr{A}(1) \modmod \euscr{A}(0))^\vee \to 0.\]
    Applying the functor $\text{Ext}_{\euscr{A}(1)^\vee}^{*,*,*}(-)$ gives a long exact sequence of $\text{Ext}_{\euscr{A}(1)^\vee}^{*,*,*}(\mathbb{M}_2^{\mathbb{F}_q})$-modules. We can use a change of rings isomorphism to rewrite the cokernel in Ext:
    \[\text{Ext}_{\euscr{A}(1)^\vee}^{*,*,*}(B_1^{\mathbb{F}_q}(\tfrac{k}{2}-1) \otimes (\euscr{A}(1)\modmod \euscr{A}(0))^\vee) \cong \text{Ext}_{\euscr{A}(0)^\vee}^{*,*,*}(B_1^{\mathbb{F}_q}(\tfrac{k}{2}-1)).\]
    As we are working mod $v_1$-torsion, the connecting homomorphism must be trivial. Thus, the Ext group in question decomposes into the Ext group of the kernel and the Ext group of the cokernel. The Ext group of the kernel,  $\text{Ext}_{\euscr{A}(1)^\vee}^{*,*,*}(\Sigma^{2k, k}B_0^{\mathbb{F}_q}(\tfrac{k}{2}))$, is handled by the induction hypothesis. A simple computation in shows that the Ext group of the cokernel consists solely of a direct sum of $\mathbb{F}_2[u, \tau, h_0]/(u^2)$ indexed along the generators of $B_1^{\mathbb{F}_q}(\tfrac{k}{2}-1)$ as an $\euscr{A}(0)^\vee$-comodule. In other words, since these summands are in the same stem as their $\mathbb{C}$-motivic counterpart \cite[Theorem 3.38]{CQ21}, we have an isomorphism
    \[\text{Ext}_{\euscr{A}(0)^\vee}^{*,*,*}(B_1^{\mathbb{F}_q}(\tfrac{k}{2}-1)) \cong \text{Ext}^{*,*,*}_{\euscr{A}(0)^\vee}(B_1^{\mathbb{C}}(\tfrac{k}{2}-1)) \otimes \mathbb{M}_2^{\mathbb{F}_q}.\]
    Altogether, we have
    \begin{align*}
    \text{Ext}^{*,*,*}_{\euscr{A}(1)^\vee}(B_0^{\mathbb{F}_q}(k)) \cong \left(\text{Ext}_{\euscr{A}(1)^\vee}^{*,*,*}(\Sigma^{2k, k}B_0^{\mathbb{C}}(\tfrac{k}{2})) \oplus \text{Ext}^{*,*,*}_{\euscr{A}(0)^\vee} (B_1^{\mathbb{C}}(\tfrac{k}{2}-1))\right) \otimes \mathbb{M}_2^{\mathbb{F}_q} \\
     \cong \text{Ext}_{\euscr{A}(1)^\vee}^{*,*,*}(B_0^\mathbb{C}(k)) \otimes \mathbb{M}_2^{\mathbb{F}_q}
    \end{align*}
    where the last equality is due to \cite[Theorem 5.5]{Realkqcoop}.

    Suppose that $k$ is odd. There is a short exact sequence of Brown--Gitler comodules
    \[0 \to \Sigma^{2(k-1), k-1}B_0^{\mathbb{F}_q}(\tfrac{k-1}{2}) \otimes B_0^{\mathbb{F}_q}(1) \to B_0^{\mathbb{F}_q}(k) \to B_1^{\mathbb{F}_q}(\tfrac{k-1}{2}-1) \otimes (\euscr{A}(1) \modmod \euscr{A}(0))^\vee \to 0.\]
    After applying $\text{Ext}_{\euscr{A}(1)^\vee}^{*,*,*}(-)$ and using a change of rings isomorphism on the cokernel, the same argument as given above implies that the connecting homomorphism in the long exact sequence of Ext groups is trivial modulo $v_1$-torsion. Thus, we only need to understand the Ext group of the kernel and the Ext group of the cokernel. The Ext group of the cokernel follows from the same style of computation as in the previous case, giving an isomorphism
    \[\text{Ext}^{*,*,*}_{\euscr{A}(0)^\vee}(B_1^{\mathbb{F}_q}(\tfrac{k-1}{2}-1)) \cong \text{Ext}^{*,*,*}_{\euscr{A}(0)^\vee}(B_1^{\mathbb{C}}(\tfrac{k-1}{2}-1)) \otimes \mathbb{M}_2^{\mathbb{F}_q}.\]
    The Ext group of the kernel may be calculated by another algebraic Atiyah--Hirzebruch spectral sequence. By applying the functor $\text{Ext}_{\euscr{A}(1)^\vee}^{*,*,*}(B_0^{\mathbb{F}_q}(\tfrac{k-1}{2}) \otimes -)$ to the cellular filtration on $B_0^{\mathbb{F}_q}(1)$, we obtain a spectral sequence with signature
    \[\mathrm{E}^{s,f,w,a}_1 = \text{Ext}^{s,f,w}_{\euscr{A}(1)^\vee}(B_0^{\mathbb{F}_q}(\tfrac{k-1}{2})) \otimes \mathbb{M}_2^{\mathbb{F}_q}\{[1], [\overline{\xi}_1], [\overline{\tau}_1]\} \implies \text{Ext}^{s,f,w}_{\euscr{A}(1)^\vee}(B_0^{\mathbb{F}_q}(\tfrac{k-1}{2}) \otimes B_0^{\mathbb{F}_q}(1)).\]
    By the induction hypothesis, we have that
    \[\text{Ext}^{*,*,*}_{\euscr{A}(1)^\vee}(B_0^{\mathbb{F}_q}(\tfrac{k-1}{2})) \cong \text{Ext}^{*,*,*}_{\euscr{A}(1)^\vee}(B_0^{\mathbb{C}}(\tfrac{k-1}{2})) \otimes \mathbb{M}_2^{\mathbb{F}_q}.\]
    The differentials are induced by the cellular filtration on $B_0^{\mathbb{F}_q}(1)$, so by the same argument as given in \Cref{prop:easyExtB01powers}, the differentials are the same as the $\mathbb{C}$-motivic case determined in \cite[Theorem 3.38]{CQ21} and there is no room for a $d_3$-differential. Thus the abutment is isomorphic to
    \[\text{Ext}_{\euscr{A}(1)^\vee}^{*,*,*}(B_0^{\mathbb{C}}(\tfrac{k-1}{2}) \otimes B_0^{\mathbb{C}}(1)) \otimes \mathbb{M}_2^{\mathbb{F}_q}.\]
    Therefore, we have
    \begin{align*}
        &\text{Ext}^{*,*,*}_{\euscr{A}(1)^\vee}(B_0^{\mathbb{F}_q}(k))  \\
         & \cong \left( \text{Ext}^{*,*,*}_{\euscr{A}^\vee}(\Sigma^{2(k-1), (k-1)}B_0^{\mathbb{C}}(\tfrac{k}{2}-1) \otimes B_0^{\mathbb{C}}(1)) \oplus \text{Ext}^{*,*,*}_{\euscr{A}(0)^\vee}(B_1(\tfrac{k-1}{2}-1))\right) \otimes \mathbb{M}_2^{\mathbb{F}_q}\\
         &\cong \text{Ext}^{*,*,*}_{\euscr{A}(1)^\vee}(B_0^{\mathbb{C}}(k)) \otimes \mathbb{M}_2^{\mathbb{F}_q}
    \end{align*}
    where the last equivalence is again due to \cite[Theorem 5.5]{Realkqcoop}.
\end{proof}

We can be more descriptive, but we must introduce some notation. For $M$ an $(s,f,w)$-trigraded module (such as an Ext group over the dual Steenrod algebra), let $\tau_{\geq n}M$ denote the truncation in the $s$-degree, so that 
\[
\tau_{\geq n}M \cong \left\{\begin{array}{lr}
    M &  s \geq n\\
    0 & \text{else}.
\end{array} \right.
\]
Let $\text{E} \in \text{SH}(\mathbb{F}_q)$ be any motivic spectrum. Whenever E has a factor of $(\euscr{A} \modmod \euscr{A}(1))^\vee$ in its mod-2 homology, we will use the notation $\text{Ext}^{*,*,*}_{\euscr{A}(1)^\vee}(\text{E})$ for the $\mathrm{E}_2$-page of the $\textbf{mASS}^{\mathbb{F}_q}(\text{E}).$ Let $\text{E}^{\langle n \rangle}$ denote the $n^{th}$ Adams cover of E, which is the $n^{th}$ term in a minimal $\text{H}\mathbb{F}_2$-Adams resolution for E.  Recall that a minimal Adams resolution is a diagram of the form \cite[Definition 2.1.3]{Rav86}
\[\begin{tikzcd}
	{\text{E}^{\langle 0 \rangle} :=\text{E}} & {\text{E}^{\langle 1 \rangle}} & {\text{E}^{\langle 2 \rangle}} & \cdots \\
	{\text{K}_0} & {\text{K}_1} & {\text{K}_2}
	\arrow[from=1-1, to=2-1]
	\arrow[from=1-2, to=1-1]
	\arrow[from=1-2, to=2-2]
	\arrow[from=1-3, to=1-2]
	\arrow[from=1-3, to=2-3]
	\arrow[from=1-4, to=1-3]
	\arrow[dashed, from=2-1, to=1-2]
	\arrow[dashed, from=2-2, to=1-3]
	\arrow[dashed, from=2-3, to=1-4]
\end{tikzcd}\]
and applying homotopy groups to this diagram gives the $\textbf{mASS}^{\mathbb{F}_q}(\text{E})$. Here $\text{K}_n := \bigoplus_B\text{H}\mathbb{F}_2$ for a basis $B$ of the homology group $\text{H}_{*,*}(\text{E}^{\langle n \rangle})$, and $\text{E}^{\langle n\rangle}$ is the fiber of the map $\text{E}^{\langle n-1\rangle} \to \text{K}_{n-1}$, defined inductively by taking $\text{E}^{\langle 0 \rangle} := \text{E}$. The group $\text{Ext}^{*,*,*}_{\euscr{A}^\vee}(\text{E}^{\langle n \rangle})$ can be described as
\[
\text{Ext}_{\euscr{A}^\vee}^{s,f,w}({\text{E}^{\langle n \rangle})) \cong \left\{\begin{array}{rl}
    \text{Ext}_{\euscr{A}^\vee}^{s,f+n, w}(\text{E)} & f \geq 0 \\
    0 & f<0.
\end{array} \right.}
\]
One may interpret this Ext group in charts by taking the charts for $\text{Ext}_{\euscr{A}^\vee}^{*,*,*}(\text{E})$, relabeling the filtration $f=n$-line as $f=0$, and then omitting anything below this new filtration $f=0$ line. In other words, one slides the $s$-axis up to filtration $f=n$, relabels, and truncates all information in lower filtration.

Notice that $\tau_{\geq n}(\text{Ext}_{\euscr{A}(1)^\vee}^{*,*,*}(\text{kq}^{\langle n\rangle}))$ and $\tau_{\geq n}(\text{Ext}^{*,*,*}_{\euscr{A}(1)^\vee}(\text{ksp}^{\langle n \rangle}))$ have isolated non-nilpotent $h_1$-towers which are not connected to an $h_0$-tower. In these cases, we will use the modified truncation $\tau_{\geq n}^{h_1}(-)$ to indicate taking the truncation and then omitting out these isolated $h_1$-towers. As an example, we depict $\tau_{\geq 4}(\text{Ext}_{\euscr{A}(1)^\vee}^{*,*,*}(\text{kq}^{\langle 2 \rangle}))$ in \Cref{fig:f5_truncation_kqadamscover2} and $\tau_{\geq 4}^{h_1}(\text{Ext}_{\euscr{A}(1)^\vee}^{*,*,*}(\text{kq}^{\langle 2 \rangle}))$ in \Cref{fig:f5_h1truncation_kqadamscover2}.

\begin{figure}
    \centering
    \includegraphics[scale=.75]{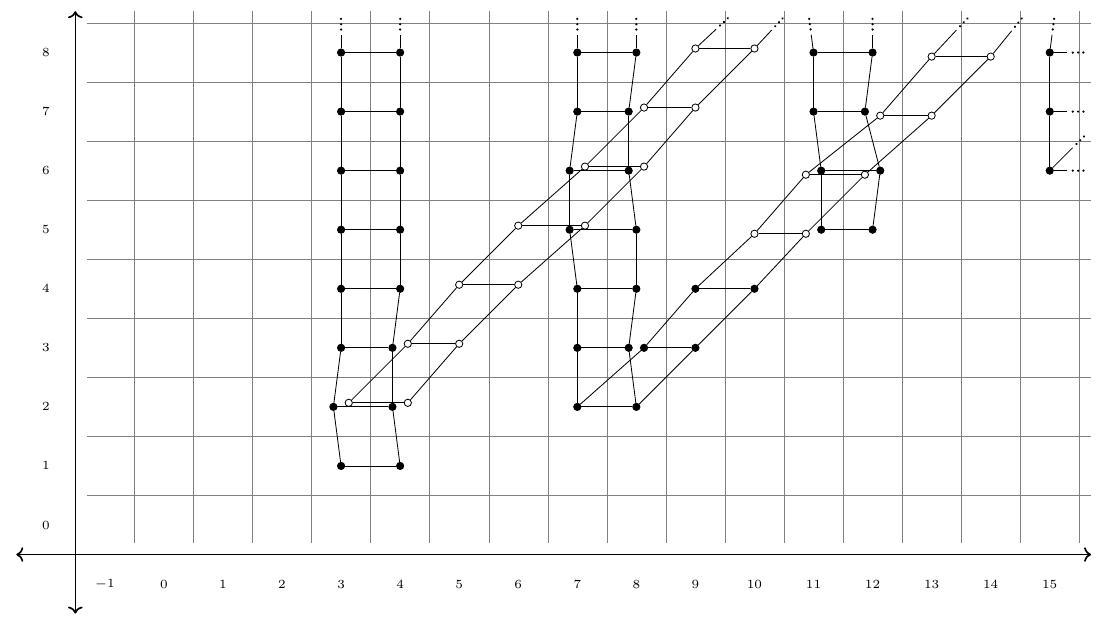}
    \caption{$\tau_{\geq 4}(\text{Ext}_{\euscr{A}(1)^\vee}^{*,*,*}(\text{kq}^{\langle 2 \rangle}))$ for $q \equiv 1 \, (4)$ with an isolated $h_1$-tower.}
    \label{fig:f5_truncation_kqadamscover2}
\end{figure}
\hfill 

\begin{figure}
    \centering
    \includegraphics[scale=.75]{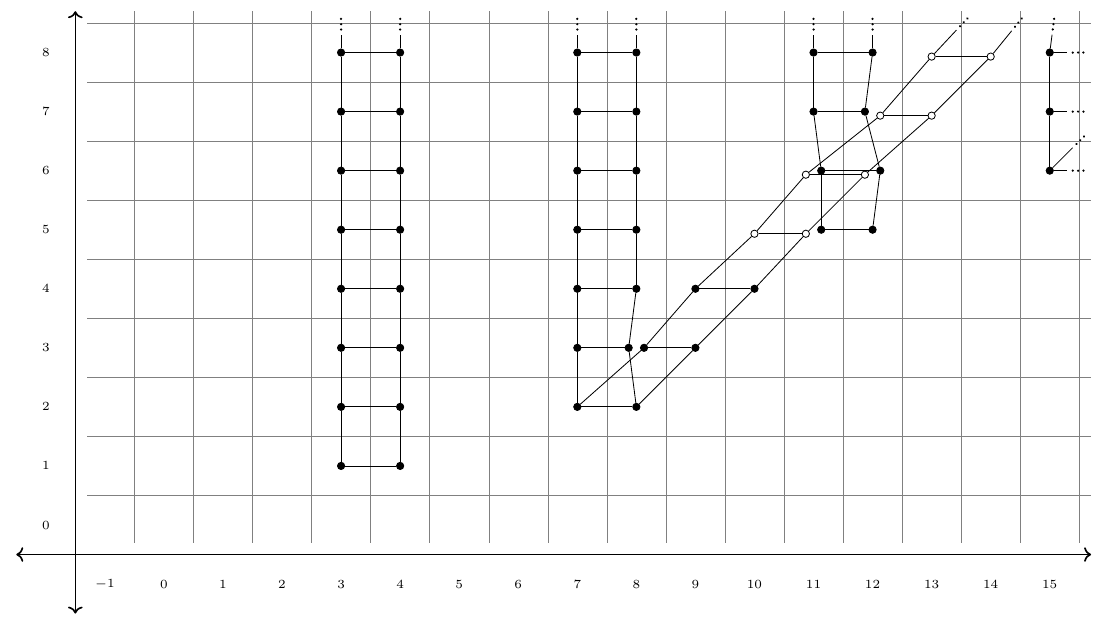}
    \caption{$\tau_{\geq 4}^{h_1}(\text{Ext}_{\euscr{A}(1)^\vee}^{*,*,*}(\text{kq}^{\langle 2 \rangle}))$ for $q \equiv 1 \, (4)$ with no isolated $h_1$-towers.}
    \label{fig:f5_h1truncation_kqadamscover2}
\end{figure}

\begin{definition}
\label{def:F5_Zi}
    For $i \geq 0$, let $Z_i^{\mathbb{C}}$ be the  $\text{Ext}^{*,*,*}_{\euscr{A}(1)^\vee}(\mathbb{M}_2^{\mathbb{C}})$-module defined as follows.
    
    For $i \equiv 0 \, (4)$, let $Z_i^\mathbb{C}$ be
    \[\tau^{h_1}_{\geq 2i}\left(\text{Ext}_{\euscr{A}(1)^\vee}^{*,*,*}(\text{kq}^{\langle i \rangle})\right) \oplus \bigoplus_{j=0}^{i/2-1}\Sigma^{4j,2j}\text{Ext}_{\euscr{A}(0)^\vee}^{*,*,*}(\mathbb{M}_2^{\mathbb{C}}).\]
    
    For $i \equiv 1 \, (4)$, let $Z_i^\mathbb{C}$ be
    \[\tau^{h_1}_{\geq 2i-2}\left(\text{Ext}_{\euscr{A}(1)^\vee}^{*,*,*}(\text{ksp}^{\langle i-1 \rangle}) \right) \oplus \bigoplus_{j=0}^{(i-1)/2-1}\Sigma^{4j,2j}\text{Ext}_{\euscr{A}(0)^\vee}^{*,*,*}(\mathbb{M}_2^{\mathbb{C}}).\]
    
    For $i \equiv 2 \, (4)$, let $Z_i^\mathbb{C}$ be
    \[\tau^{h_1}_{\geq 2i}\left(\text{Ext}_{\euscr{A}(1)^\vee}^{*,*,*}(\text{kq}^{\langle i \rangle})\right) \oplus \bigoplus_{j=0}^{i/2-1}\Sigma^{4j,2j}\text{Ext}_{\euscr{A}(0)^\vee}^{*,*,*}(\mathbb{M}_2^{\mathbb{C}}) \oplus \Sigma^{2i-2, i}\mathbb{F}_2[\tau]. \]
   
    For $i \equiv 3 \, (4)$, let $Z_i^\mathbb{C}$ be
    \[\tau^{h_1}_{\geq 2i-2}\left(\text{Ext}_{\euscr{A}(1)^\vee}^{*,*,*}(\text{ksp}^{\langle i-1 \rangle})\right) \oplus \bigoplus_{j=0}^{(i-1)/2-1}\Sigma^{4j,2j}\text{Ext}_{\euscr{A}(0)^\vee}^{*,*,*}(\mathbb{M}_2^{\mathbb{C}}) \oplus \Sigma^{2i-3, i-2}\mathbb{F}_2[\tau, h_1]/(h_1^2). \]  
\end{definition}

\begin{corollary}
\label{cor:f5ExtB0kDescriptive}
    Let $q \equiv 1 \, (4)$. For all $k \geq 0$, there is an isomorphism:
    \[\frac{\textup{Ext}^{*,*,*}_{\euscr{A}(1)^\vee}(B_0^{\mathbb{F}_q}(k))}{v_1\textup{-torsion}} \cong \left(\Sigma^{4k-4, 2k-2}Z_{\alpha(k)}^{\mathbb{C}} \oplus \bigoplus_{j =0}^{k-2}\mathbb{M}_2^{\mathbb{C}}[h_0]\right) \otimes_{\mathbb{M}_2^{\mathbb{C}}} \mathbb{M}_2^{\mathbb{F}_q},\]
    where $\alpha(k)$ is the number of 1's in the dyadic expansion of $k$, and the right hand summand is taken to be empty in the case of $k=1$.
\end{corollary}

\begin{proof}
    By \cite[Theorem 5.5]{Realkqcoop}, there is an isomorphism
    \[\frac{\textup{Ext}^{*,*,*}_{\euscr{A}(1)^\vee}(B_0^{\mathbb{C}}(k))}{v_1\textup{-torsion}} \cong \Sigma^{4k-4, 2k-2}Z_{\alpha(k)}^{\mathbb{C}} \oplus \bigoplus_{j =0}^{k-2}\mathbb{M}_2^{\mathbb{C}}[h_0] .\]
    The result then follows by \cref{thm:easyExtB0k}
\end{proof}

\begin{remark}
    The notation in \Cref{def:F5_Zi}, although equivalent, differs from the notation in \cite[Section 4]{Realkqcoop}. We hope that this presentation in terms of Adams covers is more well suited for generalization.
\end{remark}

\begin{remark}
\label{remark:FailureForSpectrumLevelSplitting}
    As is true over $\mathbb{C}$ and $\mathbb{R}$, it is in this computation that we deviate from the classical story. Mahowald is able to express the groups $\text{Ext}^{*,*}_{\euscr{A}(1)^\vee}(B_0(1)^{\otimes i})$ solely in terms of Adams covers of bo and bsp \cite{LM87}. The analogous statement in motivic homotopy is not true. If one were to express $\text{Ext}^{*,*,*}_{\euscr{A}(1)^\vee}(B_0^{\mathbb{F}_q}(1)^{\otimes i})$ in terms of Adams covers of kq or ksp, then one would see summands of non-nilpotent $h_1$-towers appearing in the formulae for $Z_i^\mathbb{C}$, reflecting that $\eta$ is non-nilpotent in $\pi_{*,*}^{\mathbb{F}_q}(\text{kq}^{\langle n \rangle})$ and $\pi_{*,*}^{\mathbb{F}_q}(\text{ksp}^{\langle n \rangle})$. However, these formulas do not allow for this, and indeed one can see through the \textbf{aAHSS} that such $h_1$-towers cannot be obtained. This is the reason behind the sophisticated truncation functors we introduced in the formulae given for $Z_i^\mathbb{C}$.
\end{remark}
\subsection{The ring of cooperations}
We now assemble our findings to compute $\pi_{*,*}^{\mathbb{F}_q}(\text{kq} \otimes \text{kq})$. First, we describe the $\textup{E}_2$-page of the $\textbf{mASS}^{\mathbb{F}_q}(\text{kq} \otimes \text{kq})$.

\begin{corollary}
    Let $q \equiv 1 \, (4)$. The $\textup{E}_2$-page of the $\textup{\textbf{mASS}}^{\mathbb{F}_q}(\textup{kq} \otimes \textup{kq})$, modulo $v_1$-torsion, is given by
    \begin{align*}
    &\bigoplus_{k \geq 0}\textup{Ext}_{\euscr{A}(1)^\vee}^{*,*,*}(\Sigma^{4k, 2k}B_0^{\mathbb{F}_q}(k)) \\ &\cong \bigoplus_{k \geq 0}\Sigma^{4k, 2k}\left(\Sigma^{4k-4, 2k-2}Z_{\alpha(k)}^{\mathbb{C}} \oplus \bigoplus_{j =0}^{4k-8}\Sigma^{4j, 2j}\mathbb{M}_2^{\mathbb{C}}[h_0]\right) \otimes \mathbb{M}_2^{\mathbb{F}_q}.\end{align*}
\end{corollary}

\begin{proof}
    The first equivalence follows from \Cref{thm:e2 mass kq coop}, and the second follows from \Cref{cor:f5ExtB0kDescriptive}.
\end{proof}

Notice that this description of the $\text{E}_2$-page is as a module over the $\text{E}_2$-page of the $\textbf{mASS}^{\mathbb{F}_q}(\text{kq})$, that is, as a module over the algebra $\text{Ext}_{\euscr{A}(1)^\vee}^{*,*,*}(\mathbb{M}_2^{\mathbb{F}_q})$. We leverage this to determine the differentials in our spectral sequence.
\begin{thm}
\label{thm:kqsmashkqdifsF5}
    The differentials in the $\textup{\textbf{mASS}}^{\mathbb{F}_q}(\textup{kq} \otimes \textup{kq})$ are determined by the differentials in the $\textup{\textbf{mASS}}^{\mathbb{F}_q}(\textup{kq}).$
\end{thm}

\begin{proof}
    The inclusion map $\text{kq} \to \text{kq} \otimes \text{kq}$ induces a map of $\text{E}_2$-pages which pushes the differentials of \Cref{prop:EasyHZDifsLiftTokq} to the $k=0$ summand of $\text{Ext}^{*,*,*}_{\euscr{A}(1)^\vee}(\text{kq} \otimes \text{kq})$. The $\text{Ext}_{\euscr{A}(1)^\vee}^{*,*,*}(\text{kq})$-module structure of $\text{Ext}^{*,*,*}_{\euscr{A}(1)^\vee}(\text{kq} \otimes \text{kq})$ and degree considerations lift all of the other differentials. 
    
    To be precise, each summand of the $\text{E}_2$-page contains a submodule which is isomorphic to (a shift of) either $\text{Ext}^{*,*,*}_{\euscr{A}(1)^\vee}(\text{kq})$ or $\text{Ext}^{*,*,*}_{\euscr{A}(1)^\vee}(\text{ksp})$. This is because each summand contains either $\text{Ext}^{*,*,*}_{\euscr{A}(1)^\vee}(\mathrm{kq}^{\langle k \rangle})$ or $\text{Ext}^{*,*,*}_{\euscr{A}(1)^\vee}(\mathrm{ksp}^{\langle k \rangle})$ by construction in \Cref{def:F5_Zi}, and each Adams cover contains a submodule isomorphic to $\text{Ext}_{\euscr{A}(1)^\vee}^{*,*,*}(\mathrm{kq})$ or $\text{Ext}_{\euscr{A}(1)^\vee}^{*,*,*}(\mathrm{ksp})$ in far enough stem by the periodicity of the corresponding Ext groups. For example, the submodule in stems $s \geq 8$ of \Cref{fig:f5_h1truncation_kqadamscover2} is isomorphic to $\text{Ext}_{\euscr{A}(1)^\vee}^{*,*,*}(\mathrm{kq}).$ Let $x$ be a generator for this submodule. Then $d_r(x)=0$ for degree reasons for all $r$. This is because all potential targets have motivic weight at least 1 lower than $x$ as they are classes of the form $u\tau^nh_0^m$ for some $m, n \geq 0$, while Adams differentials preserve motivic weight weight. This is the same reason as why $d_r(1) = 0$ for all $r$ in the $\textbf{mASS}^{\mathbb{F}_q}(\text{kq})$ or the $\textbf{mASS}^{\mathbb{F}_q}(\text{ksp})$. The rest of the differentials on this submodule are now determined by the Liebniz rule and the underlying differentials in the $\textbf{mASS}^{\mathbb{F}_q}(\mathrm{kq})$, following the proof of \Cref{prop:EasyHZDifsLiftTokq}.
    
    If $y$ is a generator for any submodule of the form $\mathbb{F}_2[u,\tau,h_0]/(u^2)$, then we must have that $d_r(y)=0$ for all $r$ as all classes in the prior stem have weight at least 1 lower than $y$. Thus, the Liebniz rule implies that all the differentials will be given by the differentials in the $\textbf{mASS}^{\mathbb{F}_q}(\text{H}\mathbb{Z})$. Finally there are no differentials on any ssummand of the form $\mathbb{F}_2[\tau]$ or $\mathbb{F}_2[\tau, h_1]/(h_1)^2$ for the same reason there is no differential on $h_1$ in the $\textbf{mASS}^{\mathbb{F}_q}(\text{kq})$, following the proof in \Cref{prop:EasyHZDifsLiftTokq}.
\end{proof}

\Cref{fig:kqsmashkqChartsF5} depicts a chart for the $\textbf{mASS}^{\mathbb{F}_q}(\text{kq} \otimes \text{kq})$, which may interpreted as either the $\mathrm{E}_2$ or $\mathrm{E}_\infty$-page. Different colors correspond to different summands of the $\text{E}_2$-page. For example, the $k=0$ summand is given by $\text{Ext}_{\euscr{A}(1)^\vee}^{*,*,*}(\text{kq})$. This is depicted in black, and starts in bidegree $(0,0)$. The $k=1$ summand is given by $\Sigma^{4,2}\text{Ext}_{\euscr{A}(1)^\vee}^{*,*,*}(\text{ksp})$. This is depicted in blue, and starts in bidegree $(4,0)$. For the $\text{E}_2$-page, let a $\blacksquare$ of any color represent $\mathbb{F}_2[u, \tau, h_0]/(u^2)$, let a $\bullet$ of any color denote $\mathbb{F}_2[\tau, u]/(u^2)$, and let a $\circ$ of any color denote $\mathbb{F}_2[u]/(u^2)$. For the $\text{E}_\infty$-page, let a $\blacksquare$ of any color represent $\text{Ext}_{\euscr{A}(0)^\vee}^{*,*,*}(\mathbb{M}_2^{\mathbb{F}_q})$ (i.e. the $\mathrm{E}_\infty$-page of the $\textbf{mASS}^{\mathbb{F}_q}(\mathrm{H}\mathbb{Z})$), and let a $\bullet$ or $\circ$ of any color be the same as on the $\mathrm{E}_2$-page. A line of slope 1 indicates $h_1$-multiplication.

\begin{figure}
    \centering
    \includegraphics[scale=.75]{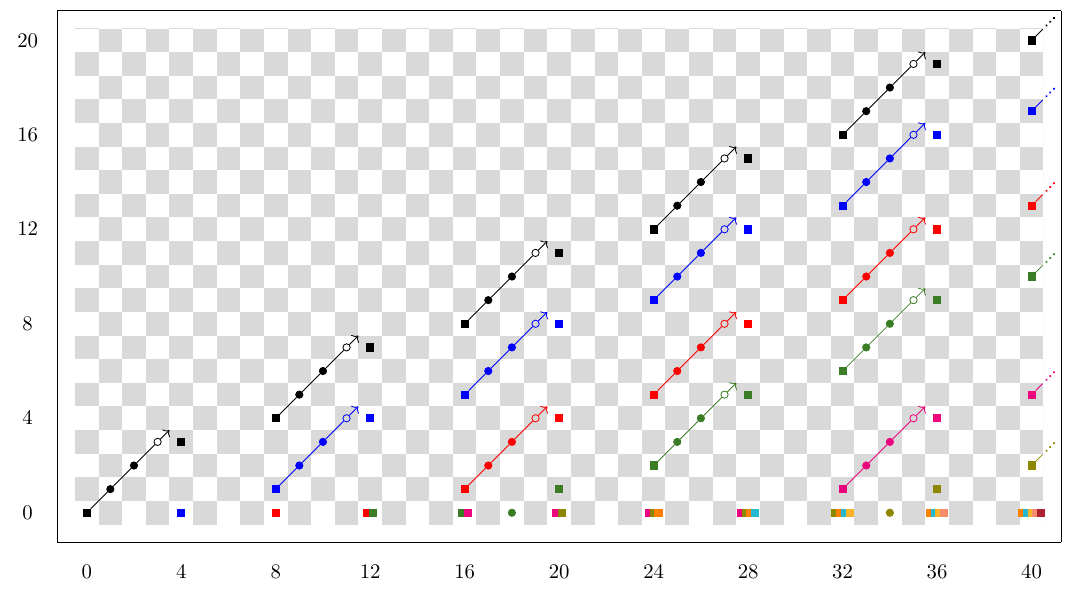}
    \caption{The $\textbf{mASS}^{\mathbb{F}_q}(\text{kq} \otimes \text{kq})$ for $q \equiv 1 \, (4).$}
    \label{fig:kqsmashkqChartsF5}
\end{figure}

\section{Nontrivial Bockstein action}
\label{section:nontrivial}
In this section, we compute the ring of cooperations $\pi_{*,*}^{\mathbb{F}_q}(\text{kq} \otimes \text{kq})$ in the case where there is a nontrivial Bockstein action, that is, when $q \equiv 3 \, (4)$. We follow the same process as was done in the previous section, and often refer to the previous section when we encounter identical proofs, but the Bockstein action makes the particular computations in Ext more complicated in this case. In particular, when computing over $\euscr{A}(1)^\vee$, we display our data in multiple charts organized by coweight.
\subsection{Integral homology}
We begin by computing the $\textbf{mASS}^{\mathbb{F}_q}(\text{H}\mathbb{Z})$. This spectral sequence has signature
\[\mathrm{E}^{s,f,w}_2 = \text{Ext}^{s,f,w}_{\euscr{A}^\vee}(\text{H}_{*,*}(\text{H}\mathbb{Z})) \implies \pi_{s,w}^{\mathbb{F}_q}(\text{H}\mathbb{Z}).\]
Recall that there is an isomorphism of $\euscr{A}^\vee$-comodules $\text{H}_{*,*}(\text{H}\mathbb{Z}) \cong (\euscr{A} \modmod \euscr{A}(0))^\vee$. This allows us to rewrite the $\mathrm{E}_2$-page of the $\textbf{mASS}^{\mathbb{F}_q}(\mathrm{H}\mathbb{Z})$ using a change of rings isomorphism as
\[\text{Ext}_{\euscr{A}^\vee}^{s,f,w}((\euscr{A} \modmod \euscr{A}(0))^\vee) \cong \text{Ext}_{\euscr{A}(0)^\vee}^{s,f,w}(\mathbb{M}_2^{\mathbb{F}_q}).\]
This Ext group was calculated by Kylling to be \cite[Theorem 4.1.3]{Kylingkqfinite}
\begin{equation}
\label{eq:ExtA0-F3}
\text{Ext}^{*,*,*}_{\euscr{A}(0)^\vee}(\mathbb{M}_2^{\mathbb{F}_q}) \cong \mathbb{F}_2[\rho, \tau^2, h_0, \rho\tau]/(\rho^2, \rho h_0,\rho(\rho\tau), (\rho\tau)^2),
\end{equation}
where $|h_0| = (1,0,0)$ and $\rho\tau$ is an indecomposable element. 

We depict this $\mathrm{E}_2$-page in \Cref{fig:f3_ext_a0}. 

\begin{remark}
    Although at first glance \Cref{fig:f3_ext_a0} looks nearly identical to \Cref{f5_Ext_A0}, it is much sparser in the motivic weight degree. For example, in stem 0 there are no classes of odd motivic weight, and in stem -1 the only classes of odd motivic weight are those of the form $\rho\tau^{2n}$. In particular, neither of the $h_0$ towers contain classes of odd motivic weight. To highlight the difference between these two sets of charts, we do not depict multiplication by $\rho\tau$.
\end{remark}

\begin{figure}
    \centering
    \includegraphics[scale=.75]{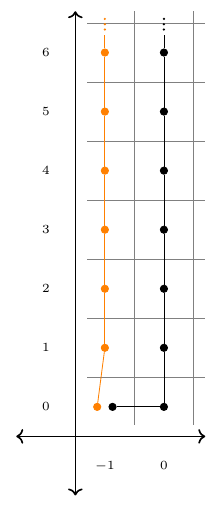}
    \caption{$\text{Ext}^{*,*,*}_{\euscr{A}(0)^\vee}(\mathbb{M}_2^{\mathbb{F}_q})$ for $q \equiv 3 \, (4).$ 
    }
    \label{fig:f3_ext_a0}
\end{figure}

We now determine all of the differentials in the spectral sequence.

\begin{lemma}[{\cite[Lemma 4.2.2]{Kylingkqfinite}}]
\label{lemma:f3HZdifs}
    Let $q \equiv 3 \, (4)$. The differentials in the $\textup{\textbf{mASS}}^{\mathbb{F}_q}(\textup{H}\mathbb{Z})$ are determined by
    \[d_{\nu_2(q^2-1)+i}(\tau^{2^i}) = \rho \tau^{2^i-1}h_0^{\nu_2(q^2-1)+i}.\]
\end{lemma}

\begin{proof}
    The same proof as was given in \Cref{F5 difs mass HZ} works by comparing with \Cref{eq:MotivicCohomology}, noting that since $\tau=0$ on the $\mathrm{E}_2$-page, the first class in stem 0 which must support a differential is $\tau^2$. 
\end{proof}

We depict the $\mathrm{E}_\infty$-page in the case of $q=3$ in \Cref{fig:f3-Einfty-HZ}, which can be deduced using the same techniques as in \Cref{ex:f5-mASS-HZ}. An empty $\circ$ denotes $\mathbb{F}_2$. A class labeled with $n$ indicates the first nonzero power of $\tau$ which exists in that particular bidegree. An empty class with $k$ circles around it indicates $\tau^{2^k}$-periodicity.
\begin{figure}
    \centering
    \includegraphics[scale=.75]{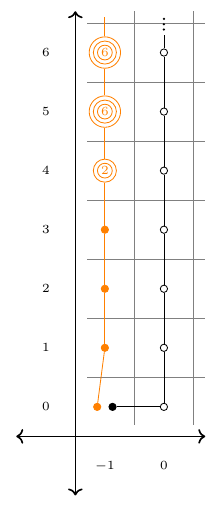}
    \caption{The $\mathrm{E}_\infty$-page of the $\textbf{mASS}^{\mathbb{F}_q}(\text{H}\mathbb{Z})$ for $q\equiv 3\, (4)$.
    }
    \label{fig:f3-Einfty-HZ}
\end{figure}

\subsection{Hermitian K-theory}
We now compute the $\textbf{mASS}^{\mathbb{F}_q}(\text{kq})$. This spectral sequence has signature
\[\mathrm{E}^{s,f,w}_2=\text{Ext}^{s,f,w}_{\euscr{A}^\vee}(\text{H}_{*,*}(\text{kq})) \implies \pi_{s,w}^{\mathbb{F}_q}(\text{kq}).\]
Recall that there is an isomorphism of $\euscr{A}^\vee$-comodules $\text{H}_{*,*}(\text{kq}) \cong (\euscr{A} \modmod \euscr{A}(1))^\vee$. This allows us to rewrite the $\mathrm{E}_2$-page of the $\textbf{mASS}^{\mathbb{F}_q}(\mathrm{kq})$ using a change of rings isomorphism as
\[\text{Ext}^{s,f,w}_{\euscr{A}^\vee}((\euscr{A}\modmod \euscr{A}(1))^\vee) \cong \text{Ext}^{s,f,w}_{\euscr{A}(1)^\vee}(\mathbb{M}_2^{\mathbb{F}_q}).\]
This Ext group was calculated by Kylling to be \cite[Theorem 4.1.5]{Kylingkqfinite}:
\begin{equation}
\label{eq:F3-E2-kq}
\textup{Ext}^{s,f,w}_{\euscr{A}(1)^\vee}( \mathbb{M}_2^{\mathbb{F}_3}) \cong \mathbb{F}_2[\rho, \tau^2, h_0, h_1, a, b, \tau h_1, \rho \tau ]/I.
\end{equation}
The degrees and coweights of the generators are listed in \cref{table:F3-ExtA(1)-generators}, where coweight refers to the difference between stem and weight. The ideal of relations $I$ is described in \cref{table:F3-ExtA(1)-relations}. 
\begin{table}[H]
\begin{minipage}{.45\linewidth}
    \centering
    \setlength{\tabcolsep}{0.5em} 
    {\renewcommand{\arraystretch}{1.2}
    \begin{tabular}{|l|l|l|}
        \hline
        Generator & $(s,f,w)$ & $cw$ \\
        \hline 
        \hline
        $\rho$ & $(-1, 0, -1)$ &0\\
        $h_0$ &$(0,1,0)$ &0\\
        $h_1$ & $(1,1,1)$&0\\
        $b$ & $(8,4,4)$&4\\
        \hline
        $\rho \tau$ & $(-1, 0, -2)$ & 1\\
        $\tau h_1$ &$(1,1,0)$ &1\\
        \hline
        $ \tau^2$ & $(0,0,-2)$&2\\
        $a$ & $(4,3,2)$& 2 \\
        \hline
    \end{tabular}
    \caption{Generators for $\text{Ext}^{*,*,*}_{\euscr{A}(1)}(\mathbb{M}_2^{\mathbb{F}_q})$ for $q \equiv 3 \, (4).$}     
    \label{table:F3-ExtA(1)-generators}}
\end{minipage}
\begin{minipage}{.45\linewidth}
    {\renewcommand{\arraystretch}{1.2}
    \begin{tabular}{|l|l|l|}
        \hline
        Relation & $(s,f,w)$ & $cw$ \\
        \hline
        \hline
        $\rho h_0$ & $(-1,1,-1)$ & 0\\
        $h_0h_1$ & $(1,2,1)$ & 0\\
        $\rho^2$ & $(-2, 0, -2)$ & 0\\
        $a^2 + h_0^2b$ & $(8,6,4)$ & 4\\
        \hline
        $h_0 \cdot \tau h_1 + \rho h_1 \cdot \tau h_1$ & $(1,2,0)$ & 1\\
        $h_1^2 \cdot \tau h_1$ & $(3,3,2)$ & 1\\
        $\rho\tau \cdot h_1^3$ & $(2,3,1)$ & 1 \\
        $\rho \cdot \tau h_1 + h_1 \cdot \rho \tau$  & $(0, 1, -1)$ & 1\\
        $\rho \cdot \rho \tau$ & $(-2, 0, -3)$ & 1\\
        \hline
        $\tau^2 h_1^3 + \rho a$ & $(3,3,1)$ & 2\\
        $h_1 a$ & $(5,4,3)$ & 2\\
        $(\tau h_1)^2 = \tau^2 h_1^2$ & $(2,2,0)$ & 2\\
        $(\rho \tau)^2$ & $(-2, 0, -4)$ & 2\\
        $\tau h_1 \cdot \rho \tau +  \rho \tau^2 h_1$ & $(0,1,-2)$ & 2\\
        \hline
        $\tau h_1 \cdot a$ & $(5,4,2)$ & 3\\
        \hline
    \end{tabular}
    \caption{Relations for $\text{Ext}^{*,*,*}_{\euscr{A}(1)^\vee}(\mathbb{M}_2^{\mathbb{F}_q})$ for $q \equiv 3 \, (4).$}      
    \label{table:F3-ExtA(1)-relations}}
\end{minipage}
\end{table}


To clearly present the data of \cref{eq:F3-E2-kq}, we give individual charts for each coweight modulo 2. Note that the element $\tau^2$ gives a coweight periodicity operator of degree 2. \Cref{fig:F3-ExtA(1)(M2) coweight 0} depicts the coweight $cw \equiv 0 \, (2)$ piece. \Cref{fig:f3-ExtA(1)(M2)-coweight 1+3} depicts the coweight $cw \equiv 1 \, (2)$ piece. The notation agrees with the notation used earlier in this section.
\begin{figure}
    \centering
    \includegraphics[scale=.75]{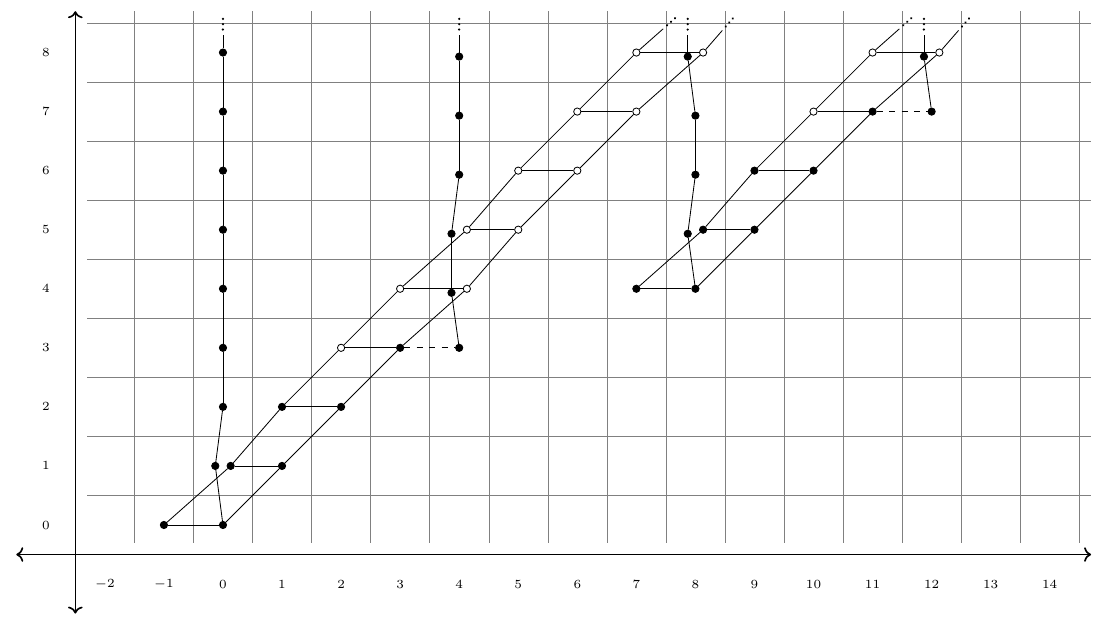}
    \caption{$\textup{Ext}^{*,*,*}_{\euscr{A}(1)^\vee}(\mathbb{M}_2^{\mathbb{F}_q})$ in coweight $cw \equiv 0 \, (2)$ for $q \equiv3 \, (4).$ 
    }
    \label{fig:F3-ExtA(1)(M2) coweight 0}
\end{figure}

\begin{figure}
    \centering
    \includegraphics[scale=.75]{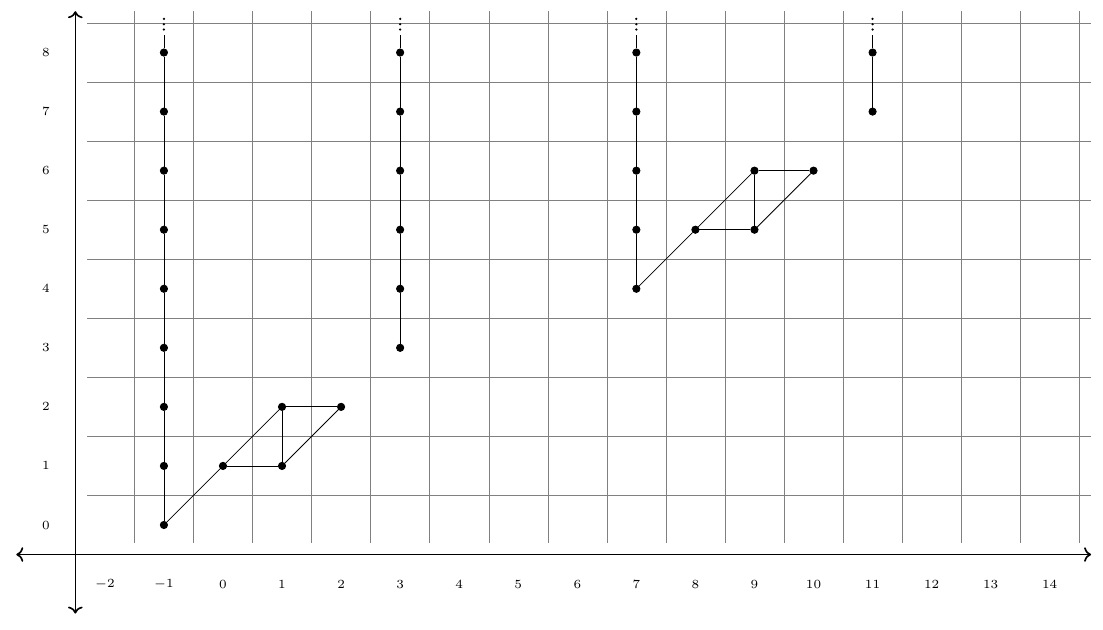}
    \caption{$\textup{Ext}^{*,*,*}_{\euscr{A}(1)^\vee}(\mathbb{M}_2^{\mathbb{F}_q})$ in coweight $cw \equiv 1 \, (2)$ for $q \equiv 3 \, (4).$ 
    }
    \label{fig:f3-ExtA(1)(M2)-coweight 1+3}
\end{figure}

As in the case of a trivial Bockstein action, we can lift the differentials from the $\textbf{mASS}^{\mathbb{F}_q}(\text{H}\mathbb{Z})$ to completely determine the behavior of this spectral sequence. The proof follows from the same argument given in \Cref{prop:EasyHZDifsLiftTokq}

\begin{proposition}
\label{prop:DifsMASSkqHard}
    Let $q \equiv 3 \, (4)$. The differentials in the $\textup{\textbf{mASS}}^{\mathbb{F}_q}(\textup{kq})$ are determined by the differentials in the $\textup{\textbf{mASS}}^{\mathbb{F}_q}(\textup{H}\mathbb{Z}).$
\end{proposition}

As in the previous section, by running the $\textbf{mASS}^{\mathbb{F}_q}(\text{kq})$ we see that 
\[\pi_{s,w}^{\mathbb{F}_q}(\text{kq}) \cong\left\{\begin{array}{rl}
    \text{KO}_{s-2w}(\mathbb{F}_q) & w \equiv 0 \, (4) \\
    \text{KSp}_{s-2w}(\mathbb{F}_q) & w \equiv 2 \ (4) 
\end{array} \right.\]
for $s \geq 0$, agreeing with Friedlander's calculations given in \Cref{table:fried}.

\begin{remark}
\label{rem:f3NegativeStemkq}
    The same argument given in \Cref{rem:f5NegativeStemkq} shows that there is an isomorphism $\pi_{s,w}^{\mathbb{F}_q}(\text{kq}) \cong \pi_{s,w}^{\mathbb{F}_q}(\text{H}\mathbb{Z})$ for $s \leq 0.$ More generally, we can express the homotopy of kq in terms of $\text{H}\mathbb{Z}$:
    \[\pi_{*,*}^{\mathbb{F}_q}(\text{kq}) \cong \frac{(\pi_{*,*}^{\mathbb{F}_q}(\text{H}\mathbb{Z}))[\eta, \tau^{n}\eta, \tau^{n}\eta^2,\alpha, \beta: n \geq 0]}{(2\eta, \eta\alpha, \tau^n\eta \alpha, \tau^n\eta^2\alpha, \alpha^2-4\beta)},\]
    where $\eta$ is detected by $h_1$, $\alpha$ is detected by $a$, and $\beta$ is detected by $b$.
    Notice that since $\tau h_1$ survives the spectral sequence, the homotopy of kq will have classes of the form $\tau^n \eta$ for all $n \geq0$, even though $\tau =0$ on the $\mathrm{E}_2$-page of the $\textbf{mASS}^{\mathbb{F}_q}(\text{kq}).$
\end{remark}

\subsection{Brown--Gitler comodules}
We begin the process of computing the $\mathrm{E}_2$-page of the $\textbf{mASS}^{\mathbb{F}_q}(\text{kq} \otimes \text{kq})$. To start, we will compute $\text{Ext}_{\euscr{A}(1)^\vee}^{s,f,w}(B_0^{\mathbb{F}_q}(1))$ by the $\textbf{aAHSS}^{\mathbb{F}_q}(B_0^{\mathbb{F}_q}(1))$. This spectral sequence has signature
\[\textup{E}_1^{s,f,w,a} = \text{Ext}^{s,f,w}_{\euscr{A}(1)}(\mathbb{M}_2^{\mathbb{F}_q}) \otimes \mathbb{M}^{\mathbb{F}_q}_2\{[1], [\overline{\xi}_1], [\overline{\tau}_1]\}\implies \text{Ext}^{s,f,w}_{\euscr{A}(1)}(B_0^{\mathbb{F}_q}(1))\]
with differentials taking the form
\[d_r:\mathrm{E}_r^{s,f,w,a} \to \mathrm{E}_r^{s-1, f+1,w, a-r}.\]
We have discussed the differentials in this spectral sequence in \Cref{inductive process section}. In particular, for any $\alpha \in \text{Ext}_{\euscr{A}(1)^\vee}^{*,*,*}(\mathbb{M}_2^{\mathbb{F}_q})$, we have that
\[d_1(\alpha[3]) = h_0\alpha[2],\]
and
\[d_2(\alpha[2]) = h_1\alpha[0].\]
There is a potential third differential between Atiyah-Hirzebruch filtrations 3 and 0, which we saw did not exist for degree reasons in the case of a trivial Bockstein action on $\mathbb{F}_q$. However, this is not the case here.

\begin{lemma}
\label{lemma:aAHSSDif}
    Let $q \equiv 3 \, (4)$. There is a $d_3$-differential in the $\textup{\textbf{aAHSS}}(B_0^{\mathbb{F}_q}(1)):$
    \[d_3(\rho[3]) = (\tau h_1)[0].\]
\end{lemma}

\begin{proof}
    This differential is given by the Massey product $\langle \rho, h_0, h_1\rangle[0]$. A simple computation in the cobar complex shows that this is equal to $\tau h_1$ with no indeterminacy. 
\end{proof}

Since the differentials in the $\textbf{aAHSS}(B_0^{\mathbb{F}_q}(1))$ are module maps over $\text{Ext}^{*,*,*}_{\euscr{A}(1)^\vee}(\mathbb{M}_2^{\mathbb{F}_q})$, this determines the spectral sequence. Hidden extensions lift as in the previous section.
Note that the $d_1$ and $d_2$-differentials are given by multiplication with an element of coweight 0, while the $d_3$-differential is given by multiplication with an element of coweight 1.

We now provide charts for the $\textup{E}_\infty$-page. \Cref{fig:f3-mASS-ksp-cw0} depicts the coweight $cw\equiv 0 \,(2)$ piece. \Cref{fig:f3-mASS-ksp-cw1_3} depicts the coweight $cw \equiv 1 \, (2)$ piece. 

\begin{figure}[h]
    \centering
    \includegraphics[scale=.75]{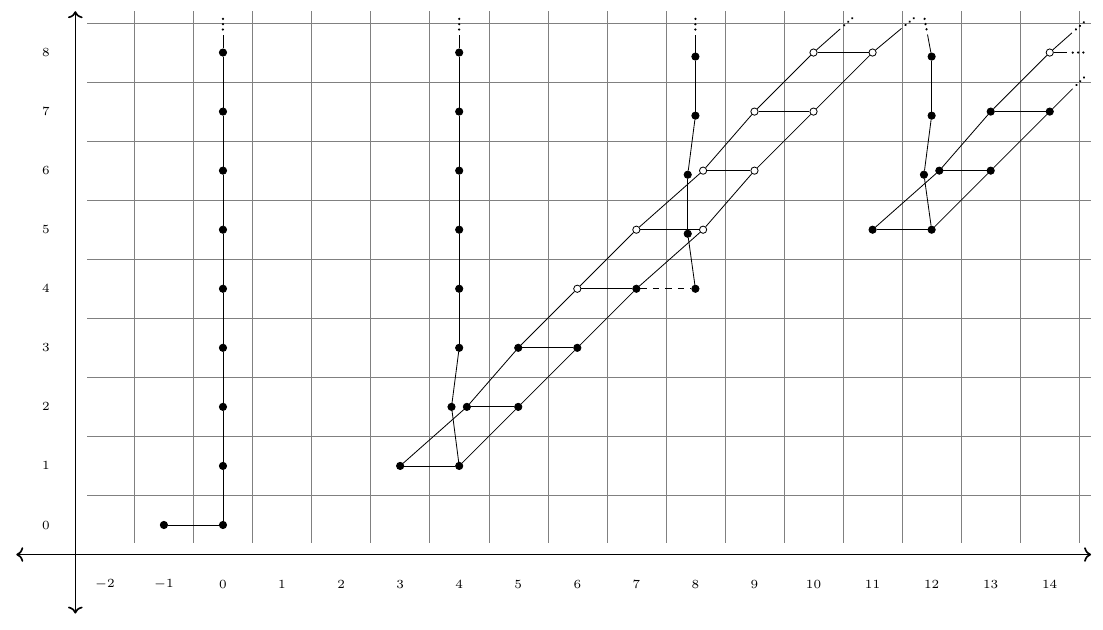}
    \caption{The $\textup{E}_\infty$-page of the $\textbf{aAHSS}(B_0^{\mathbb{F}_q}(1))$ in coweight $cw \equiv 0 \, (2)$ for $q \equiv 3\, (4)$.}
    \label{fig:f3-mASS-ksp-cw0}
\end{figure}

\begin{figure}[h]
    \centering
    \includegraphics[scale=.75]{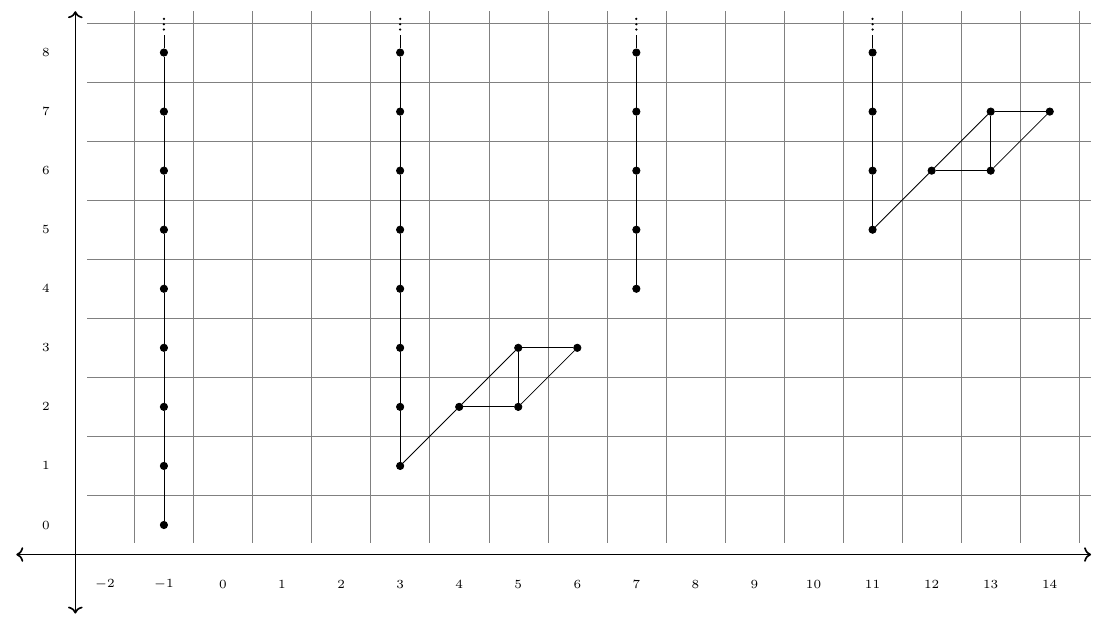}
    \caption{The $\textup{E}_\infty$-page of the $\textbf{aAHSS}(B_0^{\mathbb{F}_q}(1))$ in coweight $cw \equiv 1 \, (2)$ for $q \equiv 3\, (4)$.}
    \label{fig:f3-mASS-ksp-cw1_3}
\end{figure}

The equivalence of motivic spectra $\text{ksp} \simeq \text{H}\mathbb{Z}_1^{\mathbb{F}_q} \otimes \text{kq}$ from \cref{eq:kspSplitting}
implies that the $\textup{E}_2$-page of the $\textbf{mASS}^{\mathbb{F}_q}(\text{ksp})$ is given by $\text{Ext}_{\euscr{A}(1)^\vee}^{*,*,*}(B_0^{\mathbb{F}_q}(1)).$ Our computations with the $\textbf{aAHSS}^{\mathbb{F}_q}(B_0^{\mathbb{F}_q}(1))$ show that 
\[\text{Ext}_{\euscr{A}(1)^\vee}^{s,f,w}(B_0^{\mathbb{F}_q}(1)) \cong \Sigma^{4,2}\text{Ext}_{\euscr{A}(1)^\vee}^{s,f,w}(\mathbb{M}_2^{\mathbb{F}_q})\langle 1 \rangle \oplus \text{Ext}_{\euscr{A}(0)^\vee}^{s,f,w}(\mathbb{M}_2^{\mathbb{F}_q}).\]
The same argument as in \Cref{lem:kspmASSDifsEasy} gives the following.
\begin{lemma}
\label{lemma:kspmASSF3}
    Let $q \equiv 3 \, (4)$. In the $\textup{\textbf{mASS}}^{\mathbb{F}_q}(\textup{ksp})$, the differentials are determined by the differentials in the $\textup{\textbf{mASS}}^{\mathbb{F}_q}(\textup{kq})$ and the $\textup{\textbf{mASS}}^{\mathbb{F}_q}(\textup{H}\mathbb{Z})$.
\end{lemma}

By running the $\textbf{mASS}^{\mathbb{F}_q}(\text{ksp})$ and comparing with \Cref{table:fried}, we recover the expected isomorphism for $s \geq 4$
\[\pi_{s,w}^{\mathbb{F}_q}(\text{ksp}) \cong\left\{\begin{array}{rl}
    \text{KSp}_{s-2w}(\mathbb{F}_q) & w \equiv 0 \, (4); \\
    \text{KO}_{s-2w}(\mathbb{F}_q) & w \equiv 2 \ (4). 
\end{array} \right.\]

We now compute the group $\text{Ext}^{*,*,*}_{\euscr{A}(1)^\vee}(B_0^{\mathbb{F}_q}(1)^{\otimes i}).$ To do this, we introduce a family of trigraded groups and use the notation described in the paragraphs preceding \Cref{def:F5_Zi}. 

\begin{definition}
    For $i \geq 0$, let $Z_i^{\mathbb{F}_q}$ be the  $\text{Ext}^{*,*,*}_{\euscr{A}(1)^\vee}(\mathbb{M}_2^{\mathbb{F}_q})$-module defined as follows.
    
    For $i \equiv 0 \, (4)$, let $Z_i^{\mathbb{F}_q}$ be
    \[ \tau^{h_1}_{\geq 2i}\left(\text{Ext}_{\euscr{A}(1)^\vee}^{*,*,*}(\text{kq}^{\langle i \rangle})\right) \oplus \bigoplus_{j=0}^{i/2-1}\Sigma^{4j,2j}\text{Ext}_{\euscr{A}(0)^\vee}^{*,*,*}(\mathbb{M}_2^{\mathbb{F}_q}).\]
    
    For $i \equiv 1 \, (4)$, let $Z_i^{\mathbb{F}_q}$
    \[\tau^{h_1}_{\geq 2i-2}\left(\text{Ext}_{\euscr{A}(1)^\vee}^{*,*,*}(\text{ksp}^{\langle i-1 \rangle}) \right) \oplus \bigoplus_{j=0}^{(i-1)/2-1}\Sigma^{4j,2j}\text{Ext}_{\euscr{A}(0)^\vee}^{*,*,*}(\mathbb{M}_2^{\mathbb{F}_q}).\]
    
    For $i \equiv 2 \, (4)$, let $Z_i^{\mathbb{F}_q}$ be
    \[\tau^{h_1}_{\geq 2i}\left(\text{Ext}_{\euscr{A}(1)^\vee}^{*,*,*}(\text{kq}^{\langle i \rangle})\right) \oplus \bigoplus_{j=0}^{i/2-1}\Sigma^{4j,2j}\text{Ext}_{\euscr{A}(0)^\vee}^{*,*,*}(\mathbb{M}_2^{\mathbb{F}_q}) \oplus \Sigma^{2i-2, i}\mathbb{F}_2[\tau^2]. \]
   
    For $i \equiv 3 \, (4)$, let $Z_i^{\mathbb{F}_q}$ be
    \[\tau^{h_1}_{\geq 2i-2}\left(\text{Ext}_{\euscr{A}(1)^\vee}^{*,*,*}(\text{ksp}^{\langle i-1 \rangle})\right) \oplus \bigoplus_{j=0}^{(i-1)/2-1}\Sigma^{4j,2j}\text{Ext}_{\euscr{A}(0)^\vee}^{*,*,*}(\mathbb{M}_2^{\mathbb{F}_q}) \oplus \Sigma^{2i-3, i-2}\mathbb{F}_2[\tau^2, h_1]/(h_1^2). \]  
\end{definition}

\begin{proposition}
\label{prop:ExtB01PowersF3}
    Let $q \equiv 3 \, (4)$. For all $i \geq 0$, there is an isomorphism of $\textup{Ext}_{\euscr{A}(1)^\vee}^{*,*,*}(\mathbb{M}_2^{\mathbb{F}_q})$-modules
    \[\frac{\textup{Ext}_{\euscr{A}(1)^\vee}^{*,*,*}(B_0^{\mathbb{F}_q}(1)^{\otimes i})}{v_1\textup{-torsion}} \cong Z^{\mathbb{F}_q}_i.\]
\end{proposition}

\begin{proof}
    The proof is similar to the proof of \Cref{prop:easyExtB01powers}. The case of $i=0$ was shown in \Cref{eq:F3-E2-kq}, and the case of $i=1$ was shown in \Cref{lemma:kspmASSF3}. In the general case, consider the algebraic Atiyah-Hirzebruch spectral sequence
    \[\textup{E}^{s,f,w,a}_1 = \text{Ext}^{s,f,w}_{\euscr{A}(1)^\vee}(B_0^{\mathbb{F}_q}(1)^{\otimes i}) \otimes \mathbb{M}_2^{\mathbb{F}_q}\{[1], [\overline{\xi}_1], [\overline{\tau}_1]\} \implies \text{Ext}^{s,f,w}_{\euscr{A}(1)^\vee}(B_0^{\mathbb{F}_q}(1)^{\otimes i+1}).\]
    The differentials are determined by the cellular filtration on $B_0^{\mathbb{F}_q}(1)$ in the exact same way as before, and hidden extensions follow as in the classical case. The result follows.
\end{proof}

Before computing the groups $\text{Ext}_{\euscr{A}(1)^\vee}^{*,*,*}(B_0^{\mathbb{F}_q}(k))$, we make a preliminary calculation.
\begin{lemma}
\label{lemma:Ext-A0-B1K-OverF3}
    Let $q \equiv 3 \, (4).$ There is an isomorphism of $\textup{Ext}_{\euscr{A}(0)^\vee}^{*,*,*}(\mathbb{M}_2^{\mathbb{F}_q})$-modules
    \[\frac{\textup{Ext}_{\euscr{A}(0)^\vee}^{*,*,*}(B_1^{\mathbb{F}_q}(k))}{v_1\textup{-torsion}} \cong  \bigoplus_{j= 0}^k\Sigma^{4j, 2j}\mathbb{F}_2[\rho, \tau^2, h_0, \rho\tau]/(\rho^2, \rho h_0, \rho(\rho \tau), (\rho\tau)^2).\]
\end{lemma}

\begin{proof}
    Note that modulo $v_1$-torsion, there is an equivalence of $\euscr{A}(0)^\vee$-comodules 
    \[\frac{B_1^{\mathbb{F}_q}(k)}{v_1\textup{-torsion}} \cong \bigoplus_{j=0}^k\Sigma^{4j, 2j}\mathbb{M}_2^{\mathbb{F}_q}.\]
    The $v_1$-torsion being disregarded is a direct sum of $\euscr{A}(0)^\vee$'s, which contribute a direct sum of $\mathbb{M}_2^{\mathbb{F}_q}$'s in Adams filtration 0, all of which are necessarily $v_1$-torsion. The result now follows from \Cref{eq:ExtA0-F3}. Alternatively, one may filter the cobar complex $C_{\euscr{A}(0)^\vee}(B_1^{\mathbb{F}_q}(k))$ by powers of $\rho$ to yield a $\rho$-Bockstein spectral sequence computing the Ext group in question. Since $\rho^2=0 \in \mathbb{M}_2^{\mathbb{F}_q}$, the only differential is $d_1(\tau) = \rho h_0$, giving the desired structure.
\end{proof}

\begin{thm}
\label{thm:Ext-A1-B0k-F3}
    Let $q \equiv 3 \, (4).$ There is an isomorphism
    \[\frac{\textup{Ext}_{\euscr{A}(1)^\vee}^{*,*,*}(B_0^{\mathbb{F}_q}(k))}{v_1\textup{-torsion}} \cong \Sigma^{4k-4, 2k-2}Z_{\alpha(k)}^{\mathbb{F}_q} \oplus \bigoplus_{j=0}^{k-2}\Sigma^{4j, 2j}\mathbb{F}_2[\tau^2, h_0, \rho\tau]/((\rho\tau)^2),
    \]
    where $\alpha(k)$ is the number of 1's in the dyadic expansion of $k$, and the righthand summand is taken to be empty in the case of $k=1$.
\end{thm}

\begin{proof}
    The same argument as given in \Cref{thm:easyExtB0k} works here as well. To be clear, one uses the short exact sequences of \Cref{prop:ses bg} to induce long exact sequences in Ext groups and induction on $k$, where the case of $k=1$ was performed in \Cref{lemma:kspmASSF3}. The result follows by noticing that, since we are working modulo $v_1$-torsion, the connecting homomorphism is trivial, and the class $\rho$  coming from $\text{Ext}_{\euscr{A}(0)^\vee}^{*,*,*}(B_1^{\mathbb{F}_q}(k))$ is $v_1$-torsion for degree reasons.
\end{proof}

\subsection{The ring of cooperations}
We now assemble our findings to compute $\pi_{*,*}^{\mathbb{F}_q}(\text{kq} \otimes \text{kq})$. First, we describe the $\textup{E}_2$-page of the $\textbf{mASS}^{\mathbb{F}_q}(\text{kq} \otimes \text{kq}).$
\begin{corollary}
    Let $q \equiv 3 \, (4)$. The $\textup{E}_2$-page of the $\textup{\textbf{mASS}}^{\mathbb{F}_q}(\textup{kq} \otimes \textup{kq})$, modulo $v_1$-torsion, is given by
    \begin{align*}
    &\bigoplus_{k \geq 0}\textup{Ext}_{\euscr{A}(1)^\vee}^{*,*,*}(\Sigma^{4k, 2k}B_0^{\mathbb{F}_q}(k)) \\ &\cong \bigoplus_{k \geq 0}\Sigma^{4k, 2k}\left(\Sigma^{4k-4, 2k-2}Z_{\alpha(k)}^{\mathbb{F}_q} \oplus \bigoplus_{j=0}^{k-2}\Sigma^{4j, 2j}\mathbb{F}_2[\tau^2, h_0, \rho\tau]/((\rho\tau)^2)\right).
    \end{align*}
\end{corollary}

\begin{proof}
    The first equivalence follows from \Cref{thm:e2 mass kq coop}, and the second follows from \Cref{thm:Ext-A1-B0k-F3}.
\end{proof}

\begin{thm}
\label{thm:kqsmashkqdifsF3}
    The differentials in the $\textup{\textbf{mASS}}^{\mathbb{F}_q}(\textup{kq} \otimes \textup{kq})$ are determined by the differentials in the $\textup{\textbf{mASS}}^{\mathbb{F}_q}(\textup{kq})$.
\end{thm}

\begin{proof}
    The same proof as was given in \Cref{thm:kqsmashkqdifsF5} works here. On each summand of the $\textup{E}_2$-page, there is a submodule given by an Adams cover of kq or ksp. The differentials on this submodule are determined by the module structure over $\text{Ext}_{\euscr{A}(1)^\vee}^{*,*,*}(\mathbb{M}_2^{\mathbb{F}_q})$. On each of the $\mathbb{F}_2[\tau^2, h_0, \rho\tau]/((\rho\tau)^2)$, the differentials are determined by the formulas for the differentials in the $\textbf{mASS}^{\mathbb{F}_q}(\text{H}\mathbb{Z})$ in the same way we determined the differentials on the $\text{Ext}_{\euscr{A}(0)^\vee}^{*,*,*}(\mathbb{M}_2^{\mathbb{F}_q})$-summand of the $\textbf{mASS}^{\mathbb{F}_q}(\text{ksp}).$
\end{proof}

\Cref{fig:f3_kqsmashkq} depicts a chart for the $\textbf{mASS}^{\mathbb{F}_q}(\text{kq} \otimes \text{kq})$, which may be interpreted as either the $\mathrm{E}_2$ or $\mathrm{E}_\infty$-page. For the $\mathrm{E}_2$-page, let a $\blacksquare$ represent $\mathbb{F}_2[\rho,\tau^2, \rho\tau, h_0]/(\rho^2, \rho(\rho\tau), (\rho\tau)^2, \rho h_0)$, let a $\bullet$ represent $\mathbb{F}_2[\rho,\tau^2]/(\rho^2)$ let a $\circ$ represent $\mathbb{F}_2[\rho]/(\rho^2)$, and let a $\blacklozenge$ represents $\rho\tau \mathbb{F}_2[\rho, \tau^2]/(\rho^2).$ Note that an $h_0$-tower supports $\rho\tau$-multiplication even though $\rho h_0 =0$.
Be aware that to be concise in our notation, we have suppressed some multiplicative data. For the $\mathrm{E}_\infty$-page, as in \Cref{fig:kqsmashkqChartsF5}, one simply replaces any $\blacksquare$ with a copy of $\pi_{*,*}^{\mathbb{F}_q}(\mathrm{H}\mathbb{Z})$.

\begin{figure}
    \centering
    \includegraphics[scale=.75]{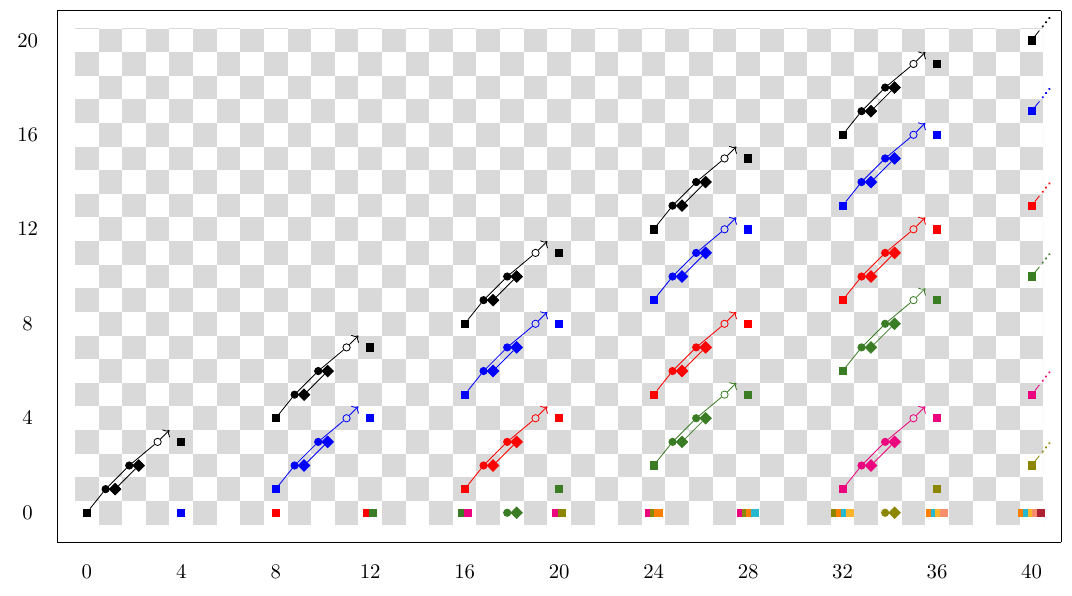}
    \caption{The $\textbf{mASS}^{\mathbb{F}_q}(\text{kq} \otimes \text{kq})$ for $q \equiv 3 \, (4)$.}
    \label{fig:f3_kqsmashkq}
\end{figure}

\begin{remark}
    We expect that the formula given for $\text{Ext}_{\euscr{A}(1)^\vee}^{*,*,*}(B_0^{\mathbb{F}_q}(k))$ generalizes well, since the short exact sequence of Brown--Gitler comodules \Cref{prop:ses bg} holds over any field with 2 invertible. We also expect that one can produce a spectrum level splitting of $\text{kq} \otimes \text{kq}$. Such a splitting would need to incorporate a mix of $h_1$-truncations of Adams covers of kq and ksp, and would need to account for the lack of non-nilpotent $h_1$-towers in low stem degrees. 
\end{remark}

\begin{remark}
    In future work, and indeed one of our initial reasons for investigating the ring of cooperations over $\mathbb{F}_q$, we will use the computation of $\pi_{*,*}^{F}(\text{kq} \otimes \text{kq})$ for $F = \mathbb{C}, \mathbb{R}, \mathbb{F}_3$ to describe $\pi_{*,*}^{\mathbb{Z}[1/2]}(\text{kq} \otimes \text{kq})$ by the methods developed in \cite{BO22}.
\end{remark}

\section{Application to the kq-resolution}
\label{section:kqres}
We conclude by examining the $\textup{E}_1$-page of the kq-resolution. In this section, we do not distinguish between the cases of a trivial or nontrivial Bockstein action on $\mathbb{F}_q$.

Recall that the kq-resolution for the sphere has signature
\[\textup{E}^{s,f,w}_1 = \pi_{s+f, w}^{\mathbb{F}_q}(\text{kq} \otimes \overline{\text{kq}}^{\otimes f}) \implies \pi_{s,w}^{\mathbb{F}_q}(\mathbb{S}).\]
We can determine each filtration $f=n$-line of the $\textup{E}_1$-page by an extension of the techniques used for the computation of $\pi_{*,*}^{\mathbb{F}_q}(\text{kq} \otimes \text{kq}).$ For each $n\geq 0$, there is a motivic Adams spectral sequence of the form
\[\textup{E}^{s,f,w}_2 = \text{Ext}^{s,f,w}_{\euscr{A}^\vee}(\text{H}_{*,*}(\text{kq} \otimes \overline{\text{kq}}^{\otimes n})) \implies \pi_{s,w}^{\mathbb{F}_q}(\text{kq} \otimes \overline{\text{kq}}^{\otimes n}).\]
We begin by decomposing the $\textup{E}_2$-page.

\begin{lemma}
\label{lemma:KunnethForkqRes}
    There is a K\"unneth isomorphism
    \[\textup{H}_{*,*}(\textup{kq} \otimes \overline{\textup{kq}}^{\otimes n}) \cong \textup{H}_{*,*}(\textup{kq}) \otimes \textup{H}_{*,*}(\overline{\textup{kq}})^{\otimes n}.\]
\end{lemma}

\begin{proof}
    We induct on $n$. The K\"unneth spectral sequence takes the form
    \[\textup{E}_2 = \text{Tor}^{\mathbb{M}_2^{\mathbb{F}_q}}(\textup{H}_{*,*}(\textup{kq}), \textup{H}_{*,*}(\overline{\textup{kq}})) \implies \textup{H}_{*,*}(\textup{kq} \otimes \overline{\textup{kq}}). \] 
    Since $\text{H}_{*,*}(\text{kq}) = (\euscr{A} \modmod \euscr{A}(1))^\vee$ is the $\mathbb{M}_2^{\mathbb{F}_q}$-linear dual of a finitely-generated free $\mathbb{M}_2^{\mathbb{F}_q}$-module, it is also free. This implies that the spectral sequence collapes. 
    
    Now, suppose the result is true for all $i <n$. We again have a  K\"unneth spectral sequence which takes the form
    \[\textup{E}_2 = \text{Tor}^{\mathbb{M}_2^{\mathbb{F}_q}}(\textup{H}_{*,*}(\textup{kq} \otimes \overline{\text{kq}}^{\otimes n-1}), \textup{H}_{*,*}(\overline{\textup{kq}})) \implies \textup{H}_{*,*}(\textup{kq} \otimes \overline{\textup{kq}}^{\otimes n}). \]
    By induction, we have a K\"unneth isomorphism on the left hand factor. Now, note that $\text{H}_{*,*}(\overline{\text{kq}})$ is also free over $\mathbb{M}_2^{\mathbb{F}_q}$, which one can see by the long exact sequence in homology associated to the cofiber sequence
    \[\mathbb{S} \to \text{kq} \to \overline{\text{kq}}.\]
    Thus the higher Tor terms vanish, implying that the spectral sequence collapses. This gives an isomorphism
    \[\text{H}_{*,*}(\text{kq} \otimes \overline{\text{kq}}^{\otimes {n-1}}) \otimes \text{H}_{*,*}(\overline{\text{kq}}) \cong \text{H}_{*,*}(\text{kq} \otimes \overline{\text{kq}}^{\otimes n}).\]
    By induction, we have a K\"unneth isomorphism on the left hand factor, finishing the proof.
\end{proof}

\begin{proposition}
\label{prop:n-lineE2}
    The $\textup{E}_2$-page of the $\textup{\textbf{mASS}}^{\mathbb{F}_q}(\textup{kq} \otimes \overline{\textup{kq}}^{\otimes n})$ may be rewritten as
    \[\textup{E}_2^{s,f,w}\cong \bigoplus_{K \in \euscr{K}_n}\Sigma^{4|K|, 2|K|}\textup{Ext}^{s,f,w}_{\euscr{A}(1)^\vee}(B_0^{\mathbb{F}_q}(K)),\]
    where $\euscr{K}_n = \{K = (k_1, \dots, k_n): k_j \geq 1 \textup{ for all } j\}$, $|K| = \sum_{j=1}^nk_j$, and $B_0^{\mathbb{F}_q}(K) = \bigotimes_{j=1}^nB_0^{\mathbb{F}_q}(k_j).$        
\end{proposition}

\begin{proof}
    By \Cref{lemma:KunnethForkqRes} and the change of rings isomorphism, we may rewrite the $\textup{E}_2$-page as
    \[\text{Ext}_{\euscr{A}^\vee}^{*,*,*}(\text{H}_{*,*}(\text{kq}) \otimes \text{H}_{*,*}(\overline{\text{kq}})^{\otimes n} ) \cong \text{Ext}^{*,*,*}_{\euscr{A}(1)^\vee}(\text{H}_{*,*}(\overline{\text{kq}})^{\otimes n}).\]
    There is an isomorphism of $\euscr{A}(1)^\vee$-comodules \cite[Proposition 2.10]{Realkqcoop}
    \[
    \text{H}_{*,*}(\overline{\text{kq}}) \cong \bigoplus_{k \geq 1}\Sigma^{4k, 2k}B_0^{\mathbb{F}_q}(k).
    \]
    This allows us to rewrite the first factor of $\text{H}_{*,*}(\overline{\text{kq}})$, leaving us with
    \[\bigoplus_{k_1 \geq 1}\Sigma^{4k_1, 2k_1}\text{Ext}_{\euscr{A}(1)^\vee}^{*,*,*}(B_0^{\mathbb{F}_q}(k_1) \otimes \text{H}_{*,*}(\text{kq})^{\otimes n-1}).\]
    Rewriting the next factor of $\text{H}_{*,*}(\overline{\text{kq}})$ gives us
    \[\bigoplus_{k_1 \geq 1}\Sigma^{4k_1, 2k_1}\left(\bigoplus_{k_2 \geq 1}\Sigma^{4k_2, 2k_2} \text{Ext}_{\euscr{A}(1)^\vee}^{*,*,*}(B_0^{\mathbb{F}_q}(k_1) \otimes B_0^{\mathbb{F}_q}(k_2) \otimes \text{H}_{*,*}(\overline{\text{kq}})^{\otimes n-2})\right),\]
    which we may rewrite as
    \[\bigoplus_{k_1, k_2 \geq 1}\Sigma^{4(k_1+k_2), 2(k_1+k_2)}\text{Ext}_{\euscr{A}(1)^\vee}^{*,*,*}(B_0^{\mathbb{F}_q}(k_1) \otimes B_0^{\mathbb{F}_q}(k_2) \otimes \text{H}_{*,*}(\overline{\text{kq}})^{\otimes n-2}).\]
    The result follows by extending these techniques and rewriting all factors of $\text{H}_{*,*}(\overline{\text{kq}})$ in terms of Brown--Gitler comodules.
\end{proof}
One may easily alter the charts given in \Cref{fig:kqsmashkqChartsF5} and \Cref{fig:f3_kqsmashkq} to illustrate the $\textbf{mASS}^{\mathbb{F}_q}(\text{kq} \otimes \overline{\text{kq}})$, i.e. the 1-line of the kq-resolution, by omiting the black $\text{Ext}^{*,*,*}_{\euscr{A}(1)^\vee}(\mathbb{M}_2^{\mathbb{F}_q})$ summand originating in stem $s=0$. For higher Adams filtration, one may compute  $\text{Ext}_{\euscr{A}(1)^\vee}^{*,*,*}(B_0^{\mathbb{F}_q}(K))$ by an analysis akin to the one we give in this paper, using a combination of the short exact sequences of Brown--Gitler comodules and algebraic Atiyah--Hirzebruch spectral sequences. In particular, while the Ext group is substantially larger,  we can still decompose it, modulo $v_1$-torsion, into:
\begin{itemize}
    \item shifts of $h_1$-truncations $\text{Ext}_{\euscr{A}(1)^\vee}^{*,*,*}(\text{kq}^{\langle n \rangle})$ and $\text{Ext}_{\euscr{A}(1)^\vee}^{*,*,*}(\text{ksp}^{\langle n \rangle})$,
    \item $\text{Ext}_{\euscr{A}(0)^\vee}^{*,*,*}(\text{H}\mathbb{Z})$-summands, and 
    \item summands of the form $(\text{Ext}_{\euscr{A}(0)^\vee}^{*,*,*}(\text{H}\mathbb{Z}))[h_1]/(h_0, h_1^2)$. 
\end{itemize}
As the 0-line of the kq-resolution is $\pi_{*,*}^{\mathbb{F}_q}(\text{kq})$, we have described the $\textbf{mASS}^{\mathbb{F}_q}(\text{kq} \otimes \overline{\text{kq}}^{\otimes n})$ as a module over the $\textbf{mASS}^{\mathbb{F}_q}(\text{kq})$. 

\begin{thm}
\label{thm:n-lineDifs}
    The differentials in the $\textup{\textbf{mASS}}^{\mathbb{F}_q}(\textup{kq} \otimes \overline{\textup{kq}}^{\otimes n})$ are determined by the differentials in the $\textup{\textbf{mASS}}^{\mathbb{F}_q}(\textup{kq}).$
\end{thm}

\begin{proof}
    In the same way that we were able to determine differentials in the $\textbf{mASS}^{\mathbb{F}_q}(\text{kq} \otimes \text{kq})$ by using the structure of $\text{Ext}_{\euscr{A}(1)^\vee}^{*,*,*}(B_0^{\mathbb{F}_q}(k))$ as a module over $\text{Ext}_{\euscr{A}(1)^\vee}^{*,*,*}(\mathbb{M}_2^{\mathbb{F}_q})$, the map 
    \[\mathbb{\text{kq}} \to \text{kq} \otimes \text{kq}^{ \otimes n}  \to \text{kq} \otimes \overline{\text{kq}} ^{\otimes n}\]
    allows us to determine the differentials in the $\textbf{mASS}^{\mathbb{F}_q}(\text{kq} \otimes \overline{\text{kq}}^{\otimes n})$ by using the structure of $\text{Ext}_{\euscr{A}(1)^\vee}^{*,*,*}(B_0^{\mathbb{F}_q}(K))$ as a module over $\text{Ext}_{\euscr{A}(1)^\vee}^{*,*,*}(\mathbb{M}_2^{\mathbb{F}_q}).$ The result follows.
\end{proof}

This gives an additive description of the $n$-line $\textup{E}_1$-page of the kq-resolution, modulo $v_1$-torsion, in terms of $h_1$-truncated Adams covers of kq and ksp, summands of $\text{H}\mathbb{Z}$, and stunted $\eta$-towers. Moreover, we have described the module structure of the $n$-line over the 0-line. We expect these observations to be useful in future work on the kq-resolution.

\begin{remark}
    In forthcoming work with Petersen and Tatum, we compute the ring of cooperations for kq and for kgl, the effective algebraic K-theory spectrum, over the $p$-adics and rationals \cite{MorPetTat}.
\end{remark}

\begin{remark}
An interesting phenomenon encountered in this paper is that any differential in a motivic Adams spectral sequence we have considered has been determined by a differential in the $\textbf{mASS}^{\mathbb{F}_q}(\text{H}\mathbb{Z})$. Similar observations were made by the author in joint work with Petersen and Tatum for the $\textbf{mASS}^{{\mathbb{F}_q}}(\text{BPGL} \langle 1 \rangle)$ and the $\textbf{mASS}^{\mathbb{F}_q}(\text{BPGL}\langle 1 \rangle \otimes \text{BPGL} \langle 1 \rangle)$ \cite{MorPetTat-BPGL1}, by Ormsby for the $\textbf{mASS}^{\mathbb{Q}_p}(\text{BPGL}\langle n \rangle)$ \cite{Ormsby11}, and by Ormsby and \O stv\ae r for the $\textbf{mASS}^{\mathbb{Q}}(\text{BPGL} \langle n \rangle)$ \cite{Ormsby-Ostvaer-motivicbp}.

Say that a 2-complete motivic spectrum E is fp if its homology $\text{H}_{*,*}(\text{E})$ is finitely-presented as an $\euscr{A}^\vee$-comodule. This is a motivic analogue of the notion of fp-spectra introduced by Mahowald--Rezk \cite{MahRez99}. Examples of motivic fp-spectra include $\text{H}\mathbb{F}_2, \text{H}\mathbb{Z}, \text{BPGL} \langle n \rangle, \text{kq},$ and, at least over $\mathbb{C}$, the connective motivic modular forms spectrum mmf \cite{GIKR22}.

\begin{question}
\label{question}
    Let $\textup{E} \in \textup{SH}(F)$ be an fp-spectrum, where $F$ is any field. To what extent are the differentials in the $\textup{\textbf{mASS}}^{F}(\textup{E})$ determined by the differentials in the $\textup{\textbf{mASS}}^F(\textup{H}\mathbb{Z})$?
\end{question}

A similar question was studied by Ormsby and \O stv\ae r \cite{Ormsby-Ostvaer-lowdim}, where they determined that over fields of low cohomological dimension, there is an isomorphism
\[\pi_{*,*}^F(\text{BPGL} \langle n \rangle) \cong (\pi_{*,*}^F(\text{H}\mathbb{Z}))[v_1 ,\dots, v_n].\]
In particular, this implies that the differentials in the $\textbf{mASS}^F(\text{BPGL} \langle n \rangle)$ are completely determined by $v_n$-linearity and the differentials in the $\textbf{mASS}^{F}(\text{H}\mathbb{Z})$. 

The above example, as well as the case of $\mathrm{kq}$, is particularly nice for comparison as the classical Adams spectral sequence for the classical analogues $\text{BP}\langle n \rangle$ and $\mathrm{bo}$ all collapse at the $\mathrm{E}_2$-page. This is not always the case for an fp-spectrum. For example, the Adams spectral sequence for the connective topological module forms spectrum $\mathrm{tmf}$ does not collapse at the $\mathrm{E}_2$-page \cite{BruRog21_ASStmf}. This implies that there are differentials in the $\textbf{ASS}(\mathrm{tmf})$ not determined by the $\textbf{ASS}(\textrm{H}\mathbb{Z})$. Over $\mathbb{C}$, the analogous statement is also true: there are no differentials in the $\textbf{mASS}^\mathbb{C}(\mathrm{H}\mathbb{Z})$, while the $\textbf{mASS}^{\mathbb{C}}(\mathrm{mmf})$ only collapses at the $\mathrm{E}_5$-page \cite{isaksen_ASS_mmf}. However, were one to construct a connective motivic modular forms spectrum over a field such as $\mathbb{F}_q$, where there are nontrivial differentials in the $\textbf{mASS}^{\mathbb{F}_q}(\mathrm{H}\mathbb{Z})$, one would expect that these differentials would also influence the differentials in the $\textbf{mASS}^{\mathbb{F}_q}(\mathrm{mmf}).$

If an fp-spectrum E is additionally a ring spectrum, then an affirmative answer to \Cref{question} implies that there is a large degree of control over the E-based motivic Adams spectral sequence. As a particular topic of investigation, we give the following.
\begin{conj}
\label{conjecture}
    Over any field $F$ of low cohomological dimension, the $\textup{E}_1$-page of the $\textup{BPGL}\langle 2 \rangle$-motivic Adams spectral sequence can be computed by appropriate lifts of the differentials in the $\textup{\textbf{mASS}}^F(\textup{H}\mathbb{Z})$. In particular, for $F = $ $\mathbb{C}$ or $\mathbb{R}$, the $\textup{\textbf{mASS}}^F(\textup{BPGL} \langle 2 \rangle^{\otimes n})$ collapses on the $\textup{E}_2$-page.
\end{conj}
\end{remark}


\printbibliography

\end{document}